\documentclass[onefignum,onetabnum]{siamart171218}

\usepackage[colorinlistoftodos]{todonotes}
\usepackage{cleveref}

\usepackage{pgfplots}
\usepackage{lipsum,multicol}
\usepackage{amsfonts}
\usepackage{graphicx}
\usepackage{epstopdf}
\usepackage{cite}
\usepackage{xspace}
\usepackage{algorithmic}
\ifpdf
  \DeclareGraphicsExtensions{.eps,.pdf,.png,.jpg}
\else
  \DeclareGraphicsExtensions{.eps}
\fi

\usepackage{amsmath,amssymb,bbm}
\usepackage[]{pstricks,pst-node,pst-text,pst-3d}
\usepackage{multimedia}
\usepackage{textpos}
\usepackage{tcolorbox}
\usepackage{extarrows}
\usepackage{mathtools}

\newsiamremark{prop}{Proposition}
\newsiamremark{remark}{Remark}
\newsiamremark{hypothesis}{Hypothesis}
\crefname{hypothesis}{Hypothesis}{Hypotheses}
\newsiamthm{claim}{Claim}

\headers{\textsf{p-AAA} for data driven modeling}{A. Carracedo Rodriguez, L. Balicki and S. Gugercin}

\title{The \paaa Algorithm for Data driven modeling of Parametric Dynamical Systems\thanks{Submitted to the editors DATE.
\funding{This work was supported in parts by National Science Foundation under Grant No. DMS-1720257 and DMS-1819110.
Part of this material is  based upon work supported by the National Science Foundation under Grant No. DMS-1439786 and by the Simons Foundation Grant No. 50736 while Gugercin was in residence at the Institute for Computational and Experimental Research in Mathematics in Providence, RI, during the ``Model and dimension reduction in uncertain and dynamic systems" program.
}}}

\author{Andrea Carracedo Rodriguez\thanks{Department of Mathematics, Virginia Tech, Blacksburg VA 24061, USA
  (\email{crandrea@vt.edu}).} \and Linus Balicki\thanks{Department of Mathematics, Virginia Tech, Blacksburg VA 24061, USA
  (\email{balicki@vt.edu}).}
\and Serkan Gugercin$^\dagger$\thanks{Department of Mathematics and
Computational Modeling and Data Analytics Division, Academy of Data Science, VA 24061, USA 
  (\email{gugercin@vt.edu}).}}

\usepackage{amsopn}

\makeatletter
\newcommand*{\addFileDependency}[1]{
  \typeout{(#1)}
  \@addtofilelist{#1}
  \IfFileExists{#1}{}{\typeout{No file #1.}}
}
\makeatother

\ifpdf
\hypersetup{
  pdftitle={The p-AAA Algorithm for Parametric Dynamical Systems},
  pdfauthor={A Carracedo Rodriguez, L Balicki, S Gugercin}
}
\fi

\newcommand{\eqdef}{\xlongequal{\text{def}}}%

\newcommand{\vf}{\textsf{VF}\xspace}
\newcommand{\aaa}{\textsf{AAA}\xspace}
\newcommand{\paaa}{\textsf{p-AAA}\xspace}
\newcommand{\mvpaaa}{\textsf{MIMO p-AAA}\xspace}
\newcommand{\bA}{\textbf{A}}
\newcommand{\bB}{\textbf{B}}
\newcommand{\bD}{\textbf{D}}

\newcommand{\bG}{\textbf{G}}
\newcommand{\bH}{\textbf{H}}
\newcommand{\bI}{\textbf{I}}
\newcommand{\bK}{\textbf{K}}
\newcommand{\bM}{\textbf{M}}
\newcommand{\bN}{\textbf{N}}

\newcommand{\ba}{\textbf{a}}
\newcommand{\bb}{\textbf{b}}
\newcommand{\bc}{\textbf{c}}
\newcommand{\be}{\textbf{e}}
\newcommand{\bff}{\textbf{f}}
\newcommand{\bg}{\textbf{g}}
\newcommand{\bh}{\textbf{h}}

\newcommand{\bu}{\textbf{u}}
\newcommand{\bv}{\textbf{v}}
\newcommand{\bx}{\textbf{x}}
\newcommand{\bw}{\textbf{w}}
\newcommand{\bJ}{\textbf{J}}
\newcommand{\bU}{\textbf{U}}
\newcommand{\bV}{\textbf{V}}
\newcommand{\hbb}{\hat{\bb}}
\newcommand{\hbc}{\hat{\bc}}
\newcommand{\hbA}{\hat{\bA}}
\newcommand{\hbE}{\hat{\textbf{E}}}

\newcommand{\Hr}{\widetilde{H}}
\newcommand{\bHr}{\widetilde{\textbf{H}}}
\newcommand{\ds}{\displaystyle}
\newcommand{\R}{\mathbb{R}}
\newcommand{\C}{\mathbb{C}}
\newcommand{\rk}{\operatorname{rank}}
\renewcommand{\Re}{\operatorname{Re}}
\renewcommand{\Im}{\operatorname{Im}}
\newcommand{\gh}{\hat{g}}
\newcommand{\bgh}{\hat{\bg}}
\newcommand{\s}{\sigma}
\newcommand{\bs}{\boldsymbol{\s}}
\newcommand{\bsh}{\hat{\bs}}
\newcommand{\sh}{\hat{\sigma}}
\newcommand{\p}{\pi} 
\newcommand{\ph}{\hat{\pi}} 
\newcommand{\boldp}{\boldsymbol{\p}}
\newcommand{\bph}{\boldsymbol{\ph}}
\newcommand{\beps}{\boldsymbol{\varepsilon}}
\newcommand{\phat}{z}
\newcommand{\z}{\zeta}
\newcommand{\zh}{\hat{\zeta}}
\newcommand{\bzeta}{\boldsymbol{\z}}
\newcommand{\bzetah}{\hat{\bzeta}}
\newcommand{\ih}{\hat{\imath}}
\newcommand{\jh}{\hat{\jmath}}
\newcommand{\F}{\bD_{\bs\boldp}} 
\newcommand{\Fhs}{\bD_{\bs\bph}} 
\newcommand{\Fhp}{\bD_{\bsh\boldp}} 
\newcommand{\Fh}{\bD_{\bsh\bph}}
\newcommand{\Loew}{\mathbb{L}}
\newcommand{\Loewhs}{\Loew_{\bs\bph}}
\newcommand{\Loewhp}{\Loew_{\bsh\boldp}}
\newcommand{\Loewh}{\Loew_{\bsh\bph}}

\newcommand{\nin}{n_{{in}}}
\newcommand{\nout}{n_{{out}}}
\newcommand{ \oldu}{f}

\newcommand{\Dpart}{
    \left[ \begin{array}{l|c}
	\F & \Fhs \\\hline
	\Fhp & \Fh
    \end{array} \right]}

\makeatletter
\def\bign#1{\mathclose{\hbox{$\left#1\vbox to8.5\p@{}\right.\n@space$}}\mathopen{}}
\def\Bign#1{\mathclose{\hbox{$\left#1\vbox to11.5\p@{}\right.\n@space$}}\mathopen{}}
\def\biggn#1{\mathclose{\hbox{$\left#1\vbox to14.5\p@{}\right.\n@space$}}\mathopen{}}
\def\Biggn#1{\mathclose{\hbox{$\left#1\vbox to17.5\p@{}\right.\n@space$}}\mathopen{}}
\makeatother
\DeclareMathOperator*{\argmax}{arg\,max}
\DeclareMathOperator*{\argmin}{arg\,min}

\def\cool#1{\textcolor{blue}{#1}}

\graphicspath{{figures/}}

\begin{document}

\maketitle

\begin{abstract}
    The \aaa algorithm has become a popular tool for  data-driven rational approximation of single variable functions, 
    such as transfer functions of linear dynamical systems. 
    In the setting of parametric dynamical systems 
    appearing in many prominent applications, 
    the underlying (transfer) function to be modeled is a multivariate function. With this in mind, 
    we develop the \aaa framework for approximating multivariate functions where the approximant is constructed in the multivariate barycentric form. 
    The method is data-driven, in the sense that it  does not require access to the full state-space model and requires only function evaluations. We discuss an extension to the case of matrix-valued functions, i.e., multi-input/multi-output dynamical systems, and provide a connection to the tangential interpolation theory. Several numerical examples illustrate the effectiveness of the proposed approach.
    \end{abstract}




\section{Introduction}
\label{sec:intro}
Many physical phenomena can be modeled as dynamical systems whose dynamics depend on one or several parameter values. These parameters
might represent material properties, boundary conditions, system geometry, etc.  
As an example, 
consider an input-output system governed by a system of {linear} ordinary differential equations  (can be viewed as a semi-discretized time-dependent PDE)
    \begin{equation} \label{paramlti}
    \dot{\bx} (t,p) = \bA (p) \bx (t,p) + \bb \oldu(t);
    \qquad
    y (t,p) = \bc^\top \bx (t,p),
    \end{equation}
where $p \in \mathcal{P} \subset \R$ 
represents the parametric variation in $\bA(p) \in \R^{\rho \times \rho}$;
$\bb,\bc \in \R^\rho$ are constant; 
$\oldu(t) \in \R$ is the input (forcing term); $y(t,p) \in \R$ is the output (quantity of interest); and $\bx(t,p)
\in \R^\rho$ 
is the state (internal degrees of freedom). 
Assuming zero initial conditions, i.e., $\bx(0) = \mathbf{0}$, the output $y(t,p)$ can be expressed using the convolution integral
    \begin{equation}  \label{paramconv}
    y (t,p) = 
    \int_0^t 
    \bc^\top e^{(t-\tau) \bA (p)} \bb \oldu(\tau) 
    ~ d\tau.
    \end{equation}
When the system dimension, $\rho$, is large,  evaluating the quantity of interest $y(t,p)$
repeatedly for different parameter values becomes computationally demanding. One remedy to this problem is to find a surrogate model of much smaller dimension, i.e., a reduced dynamical system, so that re-evaluations of the system are significantly cheaper yet accurately captures 
$y(t,p)$.
This is the goal of parametric model order reduction (PMoR).
Projection-based PMoR methods have been successfully developed for systems with known internal description as in \eqref{paramlti}, i.e., the full-order operators $\bA(p),\bb$ and $\bc$ are
available; see, e.g., the recent survey papers and books \cite{BenCOW17,QuaMN16,hesthaven2016certified,AntBG20} for a detailed analysis of projection-based approaches to PMoR.
However, in many cases {the} internal description of a system is not accessible
and 
only input/output measurements  are available. 
In our setting,
for  parametric dynamical 
systems such as \eqref{paramlti}, input/output measurements/data will correspond to the samples of \emph{the  transfer function} of \eqref{paramlti}, i.e., the samples of  
    \begin{equation} \label{Hinss}
    H (s,p) = \bc^\top (s\bI - \bA (p))^{-1} \bb,
    \end{equation}
    where $H (s,p)$ is the Laplace transform of the convolution kernel $h(t) = \bc^\top e^{t \bA(p)}\bb$ in
    \eqref{paramconv}.
Then, given the samples $\{H (s_i,p_j)\}$,  our goal is to build a function that approximates this data in an appropriate measure. Even though our motivation comes from approximating parametric dynamical systems,  similar approximation problems can also arise in modeling stationary PDEs, such as 
    \[
    u_{xx} + p u_{yy} + zu = f (x,y) 
    \qquad \textup{on} \qquad \Omega = [a,b]\times [c,d],
    \]
    with appropriately defined initial and boundary conditions.  A spatial discretization on $\Omega$, yields
    \[
    \bA (p,z) \bu = \bb.
    \]
Then, the samples of the function $ H (p,z) = \bA(p,z)^{-1} \bb$ can be used to build an approximation to the solution $u(x,y)$. We visit two such problems in \Cref{ex:jiahua}. {Assume, for the moment, that
$\bA(p)$ in~\eqref{Hinss} has an affine dependence on $p$, e.g., $\bA(p) = \bA_0 + p \bA_1$ where $\bA_0$ and $\bA_1$ are constant matrices. Then, both  $H(s,p)$  and $H(p,z)$ defined above
are two-variable rational functions. That is, $H(s,p)$ (and similarly $H(p,z)$) can be expressed as a ratio of two-variable polynomials
    \[
    H(s,p) = \frac{\sum_{i=0}^k\sum_{j=0}^{q} \tilde{\beta}_{ij}s^ip^j}{\sum_{i=0}^k\sum_{j=0}^{q} \tilde{\alpha}_{ij}s^ip^j}, \quad \tilde{\alpha}_{kq} \neq 0 \text{ or } \tilde{\beta}_{kq} \neq 0.
    \]
We refer to the tuple $\left( k, q \right)$ as the \textit{order} of $H(s,p)$. Further, we call $H(s,p)$ proper if $\tilde{\alpha}_{kq} \neq 0$ and $\tilde{\beta}_{kq} \neq 0$ and strictly proper if $\tilde{\alpha}_{kq} \neq 0$ and $\tilde{\beta}_{kq} = 0$. Even though in our approach below we do not require $H(s,p)$ to be a two-variable rational function in $(s,p)$ (and thus, we do not require
$\bA(p)$ to have an affine dependence on $p$), this form motivates us to enforce a rational form in the approximant (as done in the classical rational approximation of single-variable functions).}

Consider a scalar-valued function $ H (s,p) $ of two variables and
assume we only have access to its samples:
\[
    H (s_i,p_j) \in \C \qquad
   \mbox{for}\quad i = 1, \dots, N ~~\mbox{and}~~
    j = 1, \dots, M.
\]
{We assume that the sampling points are given and fixed, i.e.,  we are not investigating how to pick $s_i$ and $p_j$.}
Our goal is, then, to find a {two-variable} rational function $ \Hr (s,p) $ that is a \emph{good} approximation of $ H (s,p)$. We will specify later how we {evaluate the quality of our approximation}.
Even though our motivation is that $ H(s,p) $  represents the transfer function of a parametric dynamical system
and we consider the variable $s$ as frequency and $p$ as the parameter, this is  not restrictive and the approach can be considered as rational approximation of a multivariate function from its samples. {Additionally, since the proposed method will be purely based on function ($H(s,p)$) samples, there are no restrictions on the type of parameter dependence in the system to approximate. Moreover, the parameter dependence can appear in other system matrices besides $\bA(p)$. }
In order to make the derivations clear, 
we first review, in \Cref{sec:aaa1}, three of the existing algorithms for  data-driven rational approximation in the single variable case: 
the Loewner framework \cite{anderson90rational,An86scalar}, the vector fitting method \cite{Gu99rational}, and the \aaa algorithm \cite{Na18AAA}.  We highlight the similarities and differences among these three approaches.  In \Cref{sec:propmethod}, we present the proposed method, the parametric \aaa algorithm (\paaa),  for data-driven modeling of parametric dynamical systems,
which  extends the \aaa algorithm \cite{Na18AAA}  to the multivariate case.
In \Cref{subsec:mimoAAAp} we show how to apply the proposed  methodology to matrix-valued functions.
Throughout  \Cref{sec:propmethod} and \Cref{subsec:mimoAAAp}, we
use various examples to illustrate the success of the new methodology.

\section{Revisiting the single variable problem}

\label{sec:aaa1}

In this section, we briefly revisit  three approaches for the single variable case  that are pertinent to our work. The single variable function {to be approximated} can be considered as the transfer function of a non-parametric dynamical system, for example. 

Consider a single variable function $ H (s) $ and assume access to its samples 
\begin{equation} \label{svdata}
    h_i = H (s_i), \qquad s_i \in \C, \qquad \mbox{for}~~i = 1, \dots, N . 
\end{equation}
The three methods we discuss will build a rational function $ \Hr (s) $ that approximates the given data by means of interpolation, least squares (LS)  minimization, or a combination of both.
A key component in each case is the barycentric representation \cite{berrut06barycentric} of a rational function, given by
    \begin{equation}
        \label{eq:bary1}
        \Hr (s) = 
	    \frac{ n (s) }{ d (s) } =
    	\frac{ \ds \sum_{i=1}^k \frac{ \beta_i }{ s - \s_i } }{ 
	    \ds \sum_{i=1}^k \frac{ \alpha_i }{ s - \s_i } },
    \end{equation}
    where $\s_i \in \C$ are the support (interpolation) points, a subset of the sampling points $\{s_1,\ldots,s_N\}$, and $\beta_i,\alpha_i \in \C$ are the weights to be determined.
The algorithms we describe will differ from each other in how they choose 
$\s_i$'s, $\alpha_i$'s, and $\beta_i$'s. {Note that multiplying the numerator and denominator of $\Hr(s)$ by $ \prod_{i=1}^k (s-\sigma_i) $ reveals that $\Hr(s)$ is indeed a rational function of degree $k-1$}.

\subsection{The barycentric rational interpolant via Loewner matrices} \label{sec:loewner1d}
Given the data (samples) in	\cref{svdata},
the Loewner approach \cite{anderson90rational,An86scalar} builds {a} rational {function $\Hr(s)$  in \cref{eq:bary1} such that $\Hr(s_i) = h_i$ for all $i=1,\ldots,N$ (assuming a rational function of degree $k-1$ with this property exists). In this case we call $\Hr(s)$ a rational interpolant.}
Partition the sampling points and the corresponding function values:
	\begin{align*}
	\{ s_1, \dots, s_N \} & 
		= \{ \s_1, \dots, \s_k \} \cup 
		\{ \sh_1, \dots, \sh_{N-k} \} , \\
	\{ h_1, \dots, h_N \} & = 
		\{ g_1, \dots, g_k \} \cup 
		\{ \gh_1, \dots, \gh_{N-k} \} .
	\end{align*}
Interpolation at $\{\s_1,\s_2,\ldots,\s_k\}$ is attained by choosing 
	\begin{equation}
	\label{eq:int}
	    \beta_i = g_i \alpha_i ,
	\end{equation}
provided $\alpha_i$'s are nonzero.
For interpolation at  $\sh_i$, for $i=1,2,\ldots,N-k$,
we set 
\begin{align*}
     H (\sh_i) - \Hr (\sh_i) = 
     \gh_i - \frac{ n (\sh_i) }{ d (\sh_i) }
     = \gh_i - \left. {\sum_{j=1}^k \frac{g_j\alpha_j}{\sh_i-\s_j}  }
     \middle/
     {\sum_{j=1}^k \frac{\alpha_j}{\sh_i-\s_j}   } 
     \right. = 
     0.
\end{align*}
Multiplying out with the denominator, we obtain
\begin{align*}
   \gh_i \sum_{j=1}^k \frac{\alpha_j}{\sh_i-\s_j} - \sum_{j=1}^k \frac{g_j\alpha_j}{\sh_i-\s_j} 
    = \sum_{j=1}^k \frac{(\gh_i-g_j)\alpha_j}{\sh_i-\s_j}
    = \be_i^\top \Loew \ba = 0,
\end{align*}
where $ \be_i \in \R^{N-k}$ denotes the $i$th unit vector,
$ \ba^\top = \left[ \alpha_1 \cdots \alpha_k \right] $, and
$ \Loew \in \C^{(N-k)\times k} $ is the Loewner matrix given by
    \begin{equation}
    \label{eq:loew1}
        \Loew =  \left[ \begin{array}{ccc}  
		\frac{ \gh_1 - g_1 }{ \sh_1 - \s_1 } & \cdots &
		\frac{ \gh_1 - g_k}{ \sh_1 - \s_k } \\
		\vdots & \ddots & \vdots \\
		\frac{ \gh_{N-k} - g_1 }{ \sh_{N-k} - \s_1 } & \cdots &
		\frac{ \gh_{N-k} - g_k }
		{ \sh_{N-k} - \s_k } 
		\end{array} \right]
		.
    \end{equation}
Hence to enforce interpolation at  $\{\sh_1,\sh_2,\ldots,\sh_{N-k}\}$, the unknown coefficient vector 
$ \ba^\top = \left[ \alpha_1 \cdots \alpha_k \right] $
is obtained by solving the linear system
	\begin{equation} \Loew \ba = \bf 0 \label{Lnull} \end{equation}
{for $\ba \neq 0$. In particular, $\ba$ can be chosen as a singular vector associated with a zero singular value of $\Loew$ (assuming such a singular value exists).} Here, we skip the details for the conditions on   $\Loew$ and its null space to guarantee the existence and uniqueness of a degree $k-1$ rational interpolant of the form \cref{eq:bary1} and  refer the reader to   \cite{An86scalar,AntBG20} for details. 
A simple case to consider  is when $N = 2k-1$. In this case, the Loewner matrix is $\Loew \in \C^{(k-1) \times k}$, with, at least, a one-dimensional nullspace. 
Considering the fact that a proper rational function of degree $k-1$ has $2k-1$ degrees of freedom (after normalization of the highest coefficient in the denominator), choosing $N=2k-1$ will yield a unique rational interpolant (under certain conditions \cite{An86scalar,AntBG20}). By introducing the  notion of the shifted Loewner matrix, in
\cite{mayo2007fsg} 
the Loewner approach has been extended to a state-formulation where the rational interpolant can be directly written in a state-space form, as in \cref{Hinss}, without forming the barycentric form. However, for the parametric problems, the barycentric formulation is the key and we refer the reader to \cite{mayo2007fsg,ALI17,AntBG20} and the references therein for the state-space based Loewner construction for modeling dynamical systems {without parameter dependencies}.

\subsection{Vector fitting for rational least-squares approximation}  \label{sec:vf}
Instead of constructing a rational {interpolant}, one can also consider {building a rational approximant by} fitting the data in a least-squares (LS) sense. Thus, given the samples \cref{svdata}, the goal is now to construct a rational function $\Hr(s)$ that {solves the LS problem
\begin{align*}
 \min_{\alpha_j,\beta_j} \sum_{i=1}^N | \Hr (s_i) - h_i |^2.
\end{align*}}
There are various approaches to solving rational LS approximation from measured data; see, e.g., \cite{hokanson2017projected,gonnet2011robust,Gu99rational,san63transfer,drmac2015vector,berljafa2017rkfit,levy59complex,hokanson2018least,mlinaric2022} and the references therein.  Due to its close connection to the barycentric form we consider here, we briefly review the vector fitting (\vf)
method of  \cite{Gu99rational}.

\vf starts with a slightly revised version of $\Hr(s)$ with the form 
\begin{equation}
    \label{eq:vfd}
    \Hr(s) = \frac{n(s)}{d(s)}= \frac{\ds \sum_{i=1}^k \frac{ \beta_i }{ s - \s_i } }{\displaystyle 1 + \sum_{i=1}^k \frac{ \ds \alpha_i }{ s - \s_i }} + d_1 + s e_1. 
\end{equation}
A fundamental difference from the interpolation framework of {\Cref{sec:loewner1d} is that $\{\s_i\}$ in 
\cref{eq:vfd} are \emph{not} a subset of sampling points, are chosen independently, and in \vf are updated at every step. The choice of
$\{\s_i\}$ in 
\cref{eq:vfd} will be clarified later.
The additional  ``$1$" in the denominator guarantees that the first term in $\Hr(s)$ is  strictly proper. The term  $d_1 + s e_1$, if needed, allows  polynomial growth around $s = \infty$, which could be necessary in approximating transfer functions corresponding to differential algebraic equations \cite{GugSW13,MehS05,benner2017model}. These details are not fundamental to the focus of this paper; therefore we skip those and assume $d_1 = e_1 = 0$. For details, we refer the reader to \cite{Gu99rational,grivet2015passive}.

Using  \cref{eq:vfd}, the LS error can be written as 
\begin{align*}  \label{lserror}
 \sum_{i=1}^N | \Hr (s_i) - h_i |^2 = 
 \sum_{i=1}^N \frac{1}{| d(s_i) |^2}
 |n(s_i) - d(s_i) h_i|^2.
\end{align*}
This is a nonlinear LS problem. Starting with an initial guess $d^{(0)}(s)$, Sanathanan and Koerner \cite{san63transfer} converts this nonlinear LS problem into a sequence of weighted linear LS problems, which we will call the \textsf{SK} iteration:
\[
    \min_{n^{(j+1)},d^{(j+1)}} \sum_{i=1}^N  \left| \frac{n^{(j+1)}(s_i) - d^{(j+1)}(s_i) h_i}{d^{(j)}(s_i)} \right|^2,~~~j=0,1,2,\ldots.
\]
Note that the problem is now linear in the unknowns $n^{(j+1)}(s)$ and $d^{(j+1)}(s)$. The \textsf{SK} iteration uses the polynomial basis for $n(s)$
and $d(s)$. \vf, instead, uses the barycentric form  \cref{eq:vfd}, which proves to be the crucial step since it allows updating $\{\s_i\}$  in each step. \vf updates $\{\s_i\}$ as the zeros of the denominator $d^{(j)}(s)$ from the previous iteration, i.e., 
$d^{(j)} (\s_i^{(j+1)}) = 0$.
{This updating procedure for $\{\s_i\}$ and a proper rescaling}  result in a sequence of unweighted linear LS minimization problems of the form
\[
    \min_{\ba^{(j+1)}} \left\| \mathcal{A}^{(j)} \ba^{(j+1)} - \bh \right\|_2,
\]
where $\bh = \left[ h_1 ~ \cdots ~ h_N \right]^\top $, $\ba=\left[\beta_1 ~\cdots~ \beta_k~ \alpha_1 ~\cdots~ \alpha_k \right]^\top$, and $\mathcal{A}^{(j)}$ is given by
\[
    \mathcal{A}^{(j)} = \left[ \begin{array}{cccccccc}
    \frac{1}{s_1-\s_1^{(j)}} & \frac{1}{s_1-\s_2^{(j)}} & \cdots & \frac{1}{s_1-\s_k^{(j)}} & \frac{-h_1}{s_1-\s_1^{(j)}} &
\frac{-h_1}{s_1-\s_2^{(j)}} & \cdots & \frac{-h_1}{s_1-\s_k^{(j)}} \\[0.3em]
\frac{1}{s_2-\s_1^{(j)}} & \frac{1}{s_2-\s_2^{(j)}} & \cdots & \frac{1}{s_2-\s_k^{(j)}} & \frac{-h_2}{s_2-\s_1^{(j)}} &
\frac{-h_2}{s_2-\s_2^{(j)}} & \cdots & \frac{-h_2}{s_2-\s_k^{(j)}} \cr
\vdots & \vdots & \vdots & \vdots & \vdots & \vdots & \vdots & \vdots \cr
\frac{1}{s_N-\s_1^{(j)}} & \frac{1}{s_N-\s_2^{(j)}} & \cdots & \frac{1}{s_N-\s_k^{(j)}} & \frac{-h_N}{s_N-\s_1^{(j)}} &
\frac{-h_N}{s_N-\s_2^{(j)}} & \cdots & \frac{-h_N}{s_N-\s_k^{(j)}}
    \end{array} \right].
\]
Note that the Loewner matrix $\Loew$ appearing in the interpolation setting of \Cref{sec:loewner1d} is now replaced with 
 $\mathcal{A}^{(j)}$, which consists of a Cauchy and a diagonally-scaled Cauchy matrix. Despite dependence on the barycentric form, there is a fundamental difference from the Loewner framework of \Cref{sec:loewner1d}: The coefficients
 $\{ \alpha_i\}$ and $\{ \beta_i\}$ in the
barycentric form are chosen independently to minimize the LS error. This is in contrast to the Loewner setting where one sets $\beta_i = h_i \alpha_i$ to enforce interpolation. Moreover,
the points $\{\sigma_i\}$ are updated at every step.
 
Convergence of \vf is an open question. Even though one can construct examples where the iteration does not converge \cite{lefteriu2013convergence}, 
its behavior in practice is more robust. When initial set $\{\s_i\}$ is chosen appropriately, the algorithm usually converges
quickly. As \vf converges, due to the updating scheme of $\{\s_i\}$, the denominator  $d^{(k)}(s)$
converges to $1$ and one obtains a pole-residue formulation for $\Hr(s)$. However, this is not needed. The algorithm can be terminated early 
with $\Hr(s)$ having the barycentric form as in \cref{eq:vfd}.
\subsection{The \aaa algorithm}
Given the samples $\{H(s_i)\}_{i=1}^N$, we have seen two frameworks for constructing {$\Hr(s)$}: the barycentric rational interpolation via Loewner matrices (\Cref{sec:loewner1d}) and the rational LS approximation via \vf (\Cref{sec:vf}). Both methods depend on the barycentric form and differ in how they choose the variables in this representation.  
The Adaptive Anderson-Antoulas (\aaa) algorithm
developed by Nakatsukasa et al. \cite{Na18AAA}
is an iterative algorithm that elegantly integrates these two frameworks (interpolation and LS) combining their strengths, leading to a powerful  framework for rational approximation.

As in \Cref{sec:loewner1d}, we partition the sampling points $\{s_i\}$ and the samples $\{ h_i\}$ into two disjoint data sets:
\begin{align}
\begin{array}{rrcccccl} \label{AAAdata}
 \mbox{sampling~points:} \hspace*{-2ex} &   \{ s_1, \dots, s_N \} \hspace{-2ex} & = &\hspace{-2ex}
        \{~\s_1, \dots, \s_k ~\} \hspace{-2ex}&\cup& \hspace{-2ex}
		\{~\sh_1, \dots, \sh_{N-k}~\} \hspace{-2ex}&\eqdef& \hspace{-2ex}
		\{~\bs ~\cup~\bsh~\}, \\
\mbox{sampled~values:} \hspace*{-2ex} &	 \{ h_1, \dots, h_N \}\hspace{-2ex} & = &\hspace{-2ex}
	    \{ ~ g_1, \dots, g_k ~\} \hspace{-2ex}&\cup &\hspace{-2ex}
		\{ ~ \gh_1, \dots, \gh_{N-k} ~\} \hspace{-2ex}&\eqdef&\hspace{-2ex}  \{~\bg ~\cup~ \bgh~\}.
		\end{array}
\end{align}
This partitioning will be clarified later.  Assume the barycentric form for $\Hr(s)$ as in \cref{eq:bary1}, which we repeat here:
 \begin{equation}
        \tag{\ref{eq:bary1}}
        \Hr (s) = 
	    \frac{ n (s) }{ d (s) } =
	    \left.
    	{ \ds \sum_{i=1}^k \frac{ \beta_i }{ s - \s_i } }
    	\middle/
    	{ 
	    \ds \sum_{i=1}^k \frac{ \alpha_i }{ s - \s_i } }. \right.
    \end{equation}

Now assume that, we want to enforce interpolation at the points $\bs$. Therefore, in \cref{eq:bary1} we  set $\beta_i = g_i \alpha_i$ for $i=1,2,\ldots,k$, as we did in \Cref{sec:loewner1d}. However, as opposed to enforcing interpolation on $\bsh$ as well, \aaa chooses the coefficients $\{ \alpha_i\}$ to minimize the LS error over the remaining sampling points $\bsh$.

As in \Cref{sec:vf}, the LS problem 
over the sampling points
$\bsh$ is nonlinear due to dependence on the denominator $d(s)$. \vf algorithm used the \textsf{SK}-iteration to convert this nonlinear LS problem to a sequence of linearized LS problems. \aaa uses a different linearization. More precisely, for the point $\sh_i$, \aaa uses the linearization
    \begin{align}
    \label{eq:min}
     H (\sh_i) - \Hr (\sh_i) 
        & =
          \gh_i - \dfrac{n(\sh_i)}{d(\sh_i)} 
        = 
      \dfrac{1}{d(\sh_i)} \left( \gh_i d (\sh_i) - n (\sh_i) \right)
        \\ & \rightsquigarrow \label{eq:simpl}
    \gh_i d (\sh_i) - n (\sh_i)
       =
       \sum_{j=1}^k \dfrac{ ( \gh_i - g_j ) \alpha_j }{ \sh_i - \s_j }
       =
        \be_i^\top \Loew \ba,
    \end{align}
where $ \Loew $ is the Loewner matrix defined as in \cref{eq:loew1} and 
$\ba =[\alpha_1~\cdots~\alpha_k]^\top$. {This means we simply drop the term $1/d(\sh_i)$ in order to compute the coefficient vector $\ba$ via the linear LS problem (over $\bsh$), namely}
\begin{equation}  \label{minLa1d}
    \min_{\|\ba\|_2 = 1} \left\| \Loew \ba\right\|_2.
\end{equation}
Before elaborating on how \aaa  partitions the data set for interpolation and LS, we point out the difference between \cref{Lnull} and \cref{minLa1d} in determining $\ba$. In the interpolation case, assuming that there exists an underlying degree $k-1$ rational interpolant, the Loewner matrix has a null space and thus we solve $\Loew \ba = 0$. On the other hand, in the case of linearized LS problem in \aaa, such a rational interpolant does not exist (consider it as too many data points and not enough degrees of freedom), and one solves the minimization problem \cref{minLa1d} by choosing $\ba$ as the right singular vector corresponding to the smallest singular value of $\Loew$.

\aaa iteratively partitions the data using a greedy search at each step. Let 
$\Hr(s)$ denote the \aaa approximant at step $k$ corresponding to the interpolation/LS data partitioning  in \cref{AAAdata}. The next sampling point, $\s_{k+1}$, to be added to interpolation set $\bs$, is determined by finding $\sh_i$ for which the current error is maximum, i.e.,
    \[ 
    \s_{k+1} = \argmax_{i=1,\ldots,N-k} \left| H(\sh_i) - \Hr (\sh_i) \right|. 
    \]
Then, the algorithm proceeds by updating the interpolation and LS data partition, setting $\beta_{k+1} = g_{k+1} \alpha_{k+1}$, and by solving \cref{minLa1d} for the updated coefficient vector.  
\aaa is terminated when either a pre-specified error tolerance or an order is achieved. We refer the reader to the original source \cite{Na18AAA} for details. We also note that 
a similar greedy search for computing interpolation points was proposed in \cite{druskin10adaptive,feng2019new} in  projection-based interpolatory model reduction and in
\cite{lefteriu2010new} in Loewner-based interpolatory modeling. 

As \aaa proceeds, 
a new column is added to $\Loew$ at every step. 
Therefore, 
assuming large  number of data points $N$, the matrix $\Loew$ in \aaa is tall and skinny, and thus generically does not have a null space. However, if $\Loew$ happens to have a nullspace after a certain iteration index,  the \aaa approximant will interpolate the full data set and coincide with the rational interpolant of \Cref{sec:loewner1d}, assuming a unique solution.

{
\begin{remark} 
    \label{rem:denominator}
    \emph{Adding $1/d(s)$ as a weight.}
    It was pointed out in \cite[\S 10]{Na18AAA} that one can introduce weighted norms in the LS problem in every step of \aaa by scaling the rows of the Loewner matrix. 
    Inspired by the \textsf{SK} iteration and \vf,
    another type of weighting can be introduced by modifying
   the linearization step \eqref{eq:simpl}  in  \aaa  as
    \begin{equation*}
        H (\sh_i) - \Hr (\sh_i) = 
        \dfrac{1}{d(\sh_i)} 
        \left( \gh_i d (\sh_i) - n (\sh_i) \right)
        \rightsquigarrow
        \dfrac{1}{d^{-}(\sh_i)} 
        \left( \gh_i d (\sh_i) - n (\sh_i) \right)
        ,
    \end{equation*}
    where $d^{-}(s)$ denotes the denominator of the \aaa approximation from the previous step, thus keeping 
    the error still linear 
    in the variables $n(s)$ and $d(s)$ to be computed.
    Then, the coefficient vector $\ba$ can be found
    by solving the
    weighted linear LS problem
        $\min_{\|\ba\|_2=1} \left\| \Delta \Loew \ba \right\|_2$,
    where $\Delta$ is a $ k \times k $ diagonal matrix with the diagonal elements $
        \Delta_{ii} = 1/d^- (\sh_i)$.
   In our numerical experiments, this revised implementation applied to various examples did not result in a significant advantage.
    The only improvement we observed, and 
    only in some cases, 
    was a reduction by one unit in the order of the rational approximation corresponding to the same error tolerance.
    Due to these numerical observations,
    we do not investigate this further here or in the multivariate case below. Note that this weighting strategy by $1/d(s)$ focuses on adding weighting during \aaa.
    In two recent works \cite{nakatsukasa2019algorithm,filip18rational} in the setting of 
    rational minimax approximation, \aaa is followed by the Lawson
    algorithm \cite{lawson1961contribution}, an iteratively weighed LS iteration, yielding the \aaa-Lawson method. The weighting in \aaa-Lawson appears in the Lawson step, not in \aaa. 
\end{remark}
}

  The \aaa algorithm has proved  very successful and has been employed in many applications including nonlinear eigenvalue problems \cite{lie18auto},
  rational minimax approximation  
  \cite{filip18rational}, and
rational approximations over disconnected domains \cite{Na18AAA}. Our goal, in the following sections, is to extend \aaa to approximating parametric (dynamical) systems from their samples.

\section{\paaa: \aaa for parametric dynamical systems}
\label{sec:propmethod}
In this section, we introduce the parametric \aaa (\textsf{p-AAA}) algorithm, which extends  \aaa to multi-variable problems  appearing in the modeling of (the transfer function of) parametric dynamical systems.  We start with the two-variable case first and illustrate its performance on various examples. Then, 
we briefly discuss how \paaa can be applied to functions with more than two variables followed by an application to such an example. In this section, to simplify the initial discussion, we only focus on scalar-valued functions. The \paaa for matrix valued functions is discussed in \Cref{subsec:mimoAAAp}.

\subsection{\paaa for the two-parameter case}
\label{sec:aaaparam}

We consider  the problem of rational approximation of a multivariate function $ H (s,p) $ from data.
We assume only access to the samples of $H(s,p)$, i.e., we have 
\begin{equation}
    \label{eq:prob}
    h_{ij} = H (s_i,p_j) \in \C \qquad
   \mbox{for}\quad i = 1, \dots, N ~~\mbox{and}~~
    j = 1, \dots, M.
\end{equation}
Analogously to the single-variable case,
we express the rational approximant 
$\Hr(s,p)$ in its \emph{two-variable} barycentric form
    \begin{equation}
    \label{eq:Hrbary}
	\widetilde{H} (s,p) = \frac{n(s,p)}{d(s,p)} = 
	\sum_{i=1}^k \sum_{j=1}^q 
	\frac{ \beta_{ij} }{ (s-\s_i) (p-\p_j) }
	\biggn/
	\sum_{i=1}^k \sum_{j=1}^q 
	\frac{ \alpha_{ij} }{ (s-\s_i) (p-\p_j) } ,
	\end{equation}
	where $\{ \s_i \}$ and 
	$\{ \p_j \}$
	are to-be-determined  points, subsets of
	$\{s_i\}$ and $\{p_j\}$, respectively; and $\beta_{ij}$ and $\alpha_{ij}$ are scalar coefficients to be chosen based on the interpolation and LS conditions to be enforced on the data \cref{eq:prob}. {Similar to the single variable case multiplying $n(s,p)$ and $d(s,p)$ by $\prod_{i=1}^{k}\prod_{j=1}^{q} (s-\sigma_i)(p-\pi_j)$ reveals that $\Hr(s,p)$ is a two-variable rational function of order $(k-1,q-1)$.}
	The number of points, $k$, in the variable-$s$
	and $q$ in the variable-$p$ will be automatically determined by the algorithm.
	
We start by partitioning the data \cref{eq:prob}:
\begin{align}
        \begin{split}
        \label{eq:Lpart}
        \{s_1, \dots, s_N\} & = \{ \s_1, \dots, \s_k \} \cup 
		    \{ \sh_1, \dots, \sh_{N-k} \}  \eqdef 
		    \{\bs ~\cup~ \bsh\},
		    \\
      \{p_1, \dots, p_M\} & = \{\p_1,\dots,\p_q\} \cup 
            \{\ph_1, \dots, \ph_{M-q}\}
             \eqdef  \{\boldp~\cup~\bph\},~\mbox{and} \\
      &
      \renewcommand{\arraystretch}{1.2} \left[ \begin{array}{c|c}
		[ H (\s_i,\p_j) ] & [ H (\s_i,\ph_j) ] \\\hline
		[ H (\sh_i,\p_j) ] & [ H (\sh_i,\ph_j) ]
		\end{array} \right] \renewcommand{\arraystretch}{1}
		\eqdef 
		\renewcommand{\arraystretch}{1.3}
		\Dpart,
		\renewcommand{\arraystretch}{1}
        \end{split}
    \end{align}	
   where 
    $ [ H (\s_i,\p_j) ] = \F$ denotes the
    $k \times q$ matrix whose $(i,j)$th  entry is 
    $H (\s_i,\p_j)$; and similarly for other quantities such as $[ H (\s_i,\ph_j) ]= \Fhs$.
    We use $\F$  to denote the sampled data corresponding to the sampling points $(\bs,\boldp)$ (and similarly for other samples)
    as opposed to $\bH_{\bs\boldp}$ since
    $\bH(s,p)$ will be used in  \Cref{subsec:mimoAAAp} to denote  matrix-valued (transfer) functions.
    How data is partitioned as  in \cref{eq:Lpart} will be clarified later.

    \subsubsection*{Interpolation of the sampled data $\F$} 
    In accordance with the partitioning of the data in \cref{eq:Lpart},
  first we enforce interpolation at $(\bs,\boldp)$, i.e., on the (1,1) block $\F$, of the sampled data. This is achieved by setting, in \cref{eq:Hrbary},
\begin{equation}
    \label{eq:int2}
    \beta_{ij} = H (\s_i,\p_j) \alpha_{ij},
\end{equation}
assuming $\alpha_{ij} \neq 0$. 
This follows from the fact that, as in the single variable case, the barycentric form $\Hr(s,p)$ in \cref{eq:Hrbary}
has a removable singularity at $(\s_i,\p_j)$ with
$\Hr(\s_i,\p_j) = \beta_{ij}/\alpha_{ij}$~\cite{An12two}, and  the choice
\eqref{eq:int2} leads to interpolation of the data in $\F$. 
This determines $\beta_{ij}$. What remains to fully specify $\Hr(s,p)$ is the choice of $\alpha_{ij}$. 

\subsubsection*{LS fit for the uninterpolated data}
The rational approximant $\Hr(s,p)$ in \cref{eq:Hrbary} with the choice \cref{eq:int2}, 
interpolates the data $\F$. Next, we show how to chose $\alpha_{ij}$ so that
$\Hr(s,p)$ minimizes the LS error in the remaining sampled data set  in $\Fhs$, $\Fhp$, and $\Fh$, i.e., to minimize 
\begin{equation}
\left \|
\beps
\right\|_2 = 
\left \|
\begin{bmatrix}
\beps_1 \\ 
\beps_2 \\
\beps_3
\end{bmatrix}
\right\|_2 \eqdef 
\left \|
\begin{bmatrix} \textsf{vec}(\Fhs) \\ \textsf{vec}(\Fhp) \\
\textsf{vec}(\Fh) \end{bmatrix}- \begin{bmatrix}\textsf{vec}(\Hr(\bs,\bph)) \\  \textsf{vec}(\Hr(\bsh,\boldp))\\ \textsf{vec}(\Hr(\bsh,\bph))\end{bmatrix}  
\right\|_2.
\label{errorindata}
\end{equation}
As in the single variable {case}, the resulting LS problem is nonlinear and we will linearize it similarly. To illustrate this more clearly, 
we rewrite the error for a sample $(\sh,\ph)$ in the set $(\bsh,\bph)$ corresponding to a component in
$\beps_3$ in \cref{errorindata} as
    \begin{align*}
        H (\sh,\ph) - \Hr (\sh,\ph)
        & =
        H (\sh,\ph) - \dfrac{n(\sh,\ph)}{d(\sh,\ph)}
        \\ & = 
        \dfrac{1}{d(\sh,\ph)} \left( H (\sh,\ph) d (\sh,\ph) - n (\sh,\ph) \right)
        \\ & \rightsquigarrow 
        H (\sh,\ph) d (\sh,\ph) - n (\sh,\ph)~~~(\mbox{linearization})
        \\ & =
        H (\sh,\ph) \sum_{i=1}^k \sum_{j=1}^q 
	\frac{ \alpha_{ij} }{ (\sh-\s_i) (\ph-\p_j) } - \sum_{i=1}^k \sum_{j=1}^q \frac{ H (\s_i,\p_j) \alpha_{ij} }{ (\sh-\s_i) (\ph-\p_j) }
        \\ & =
        \sum_{i=1}^k \sum_{j=1}^q \dfrac{ ( H(\sh,\ph) - H(\s_i,\p_j ) )  \alpha_{ij} }{ (\sh - \s_i)(\ph-\p_j) }
        \\ & =
        \be_{{\tiny \sh\ph}}^\top\, \Loewh\, \ba,
    \end{align*}
where  \begin{equation} \label{vectora} 
\ba^\top = [ \alpha_{11} \cdots \alpha_{1q} ~|~ \cdots ~|~ \alpha_{k1} \cdots \alpha_{kq} ] \in \C^{k q},\end{equation}
$ \Loewh \in \C^{(N-k)(M-q) \times (kq)} $ is the 2D Loewner matrix\footnote{Similar to the single-variable case, the Loewner matrices appearing in \paaa here also appear in the parametric Loewner framework 
\cite{An12two,Io14data} where one aims to interpolate the full data set. We revisit these 
connections in \cref{paramLoewremark}.} defined by
\begin{align}
    \begin{split}
    \label{eq:loew2}
    \Loewh & = 
    \left[ 
    \begin{array}{ccc|c}
         \frac{H(\sh_1,\ph_1)-H(\s_1,\p_1)}{(\sh_1-\s_1)(\ph_1-\p_1)} & 
        \cdots & 
        \frac{H(\sh_1,\ph_1)-H(\s_1,\p_q)}{(\sh_1-\s_1)(\ph_1-\p_q)} &
        \cdots \\
        & \vdots & & \\
        \frac{H(\sh_{N-k},\ph_{M-q})-H(\s_1,\p_1)}{(\sh_{N-k}-\s_1)(\ph_{M-q}-\p_1)} & 
        \cdots & 
        \frac{H(\sh_{N-k},\ph_{M-q})-H(\s_1,\p_q)}{(\sh_{N-k}-\s_1)(\ph_{M-q}-\p_q)} & \cdots
    \end{array}
    \right. \\[2ex]
    & \hspace{2cm} \left.
    \begin{array}{c|ccc}
         \cdots &
        \frac{H(\sh_1,\ph_1)-H(\s_k,\p_1)}{(\sh_1-\s_k)(\ph_1-\p_1)} &
        \cdots &
        \frac{H(\sh_1,\ph_1)-H(\s_k,\p_q)}{(\sh_1-\s_k)(\ph_1-\p_q)} \\
        & & \vdots & \\
        \cdots & 
        \frac{H(\sh_{N-k},\ph_{M-q})-H(\s_k,\p_1)}{(\sh_{N-k}-\s_k)(\ph_{M-q}-\p_1)} &
        \cdots &
        \frac{H(\sh_{N-k},\ph_{M-q})-H(\s_k,\p_q)}{(\sh_{N-k}-\s_k)(\ph_{M-q}-\p_q)}
    \end{array}
    \right],
    \end{split}
\end{align}
and $\be_{{\tiny \sh\ph}} \in \R^{(N-k)(M-q)}$ is the unit vector with $1$ in the entry corresponding to the sample $(\sh,\ph)$. Therefore, the linearized error $\beps_3$
is given by $\Loewh\,\ba$. {Note that $\Loewh$ has a nested structure that takes the differences of all combinations of samples into consideration. The entries are explicitly given by:
$$
 \Loewh(\jh + (M-q)(\ih - 1), j + q(i-1)) = \frac{H(\sh_{\ih},\ph_{\jh})-H(\s_i,\p_j)}{(\sh_{\ih}-\s_i)(\ph_{\jh}-\p_j)}, \; 
$$
for $\jh = 1,\ldots,M-q$, $\ih = 1,\ldots,N-k$, $j=1,\ldots,q$, and $i=1,\ldots,k$.}

The procedure follows similarly for the other blocks  in \cref{errorindata}. {First note that
$$
\Hr(\s_i,\ph_\ell) = \left.\sum_{j=1}^{q}\frac{\beta_{ij}}{\ph_\ell - \p_j} \middle /{\sum_{j=1}^{q}\frac{\alpha_{ij}}{\ph_\ell - \p_j}}\right..
$$
}
{This expression together with the definition of $\beta_{ij}$ in \cref{eq:int2}
allow us to write the} error corresponding to a sample $(\s_i,\ph_\ell)$ in $\beps_1$
in \cref{errorindata} as 
\begin{align*}
    H(\s_i,\ph_\ell) - \Hr(\s_i,\ph_\ell) & = 
    \left(
    \left. {\ds \sum_{j=1}^q 
    \frac{H(\s_i,\ph_\ell)-H(\s_i,\p_j)}{\ph_\ell-\p_j} \alpha_{ij} } \right)
    \middle/
    {
    \ds \sum_{j=1}^q \frac{\alpha_{ij}}{\ph_\ell-\p_j} 
    }
    \right. \\
    & 
 \rightsquigarrow
    \sum_{j=1}^q \frac{H(\s_i,\ph_\ell)-H(\s_i,\p_j)}{\ph_\ell-\p_j} \alpha_{ij}~~~{(\mbox{linearization})}\\
    & = \be_\ell^\top \Loew_{\s_i} \ba_i,
\end{align*}
{where}
$ \ba_i^\top = [ \alpha_{i1} \cdots \alpha_{iq} ]\in \C^{q}$
is the $i$th row block of $\ba$,
$\be_\ell \in \C^{M-q}$ is the $\ell$th unit vector, 
and
 $\Loew_{\s_i} \in \C^{(M-q)\times q}$ is the regular (1D) Loewner matrix corresponding to the data in the $i$th row of $[\F ~ \Fhs]$, i.e.,
\begin{equation} \label{Lsigmailj}
    ( \Loew_{\s_i})_{\ell,j} = 
\frac{H(\s_i,\ph_\ell)-H(\s_i,\p_j)}{\ph_\ell-\p_j}~~\mbox{for}~~\ell = 1,2,\ldots,M-q~~\mbox{and}~~j = 1,2,\ldots,q.
\end{equation}
Similar to \cite{Io14data}, define \begin{equation}
\label{Lsigma}
    \Loewhs = \textsf{diag}(\Loew_{\s_1},\ldots,\Loew_{\s_k})
\in \C^{(k (M-q)) \times (kq)}.
\end{equation}
Then, the linearized error corresponding to $\beps_1$
in \cref{errorindata} is given by $\Loewhs \ba$.
Similarly, 
we can linearize and rewrite the error for the $\beps_2$-block in \cref{errorindata} 
as $\Loewhp\ba$ where $\Loewhp$ is an assembly of all 1D Loewner matrices $\Loew_{\p_j}$ corresponding to the data in each column of $\left[\begin{array}{l}\F\\\Fhp\end{array}\right]$.
Putting all three together, after linearization, minimizing the LS error  \cref{errorindata} in \paaa becomes
\begin{equation}
\min_{\| \ba\|_2 = 1}
\|\Loew_2\ba\|_2\quad \mbox{where}\quad
    \label{eq:l2mat}
    \Loew_2 = \left[ 
    \Loewhs^\top ~~ \Loewhp^\top ~~ \Loewh^\top
    \right]^\top {\in \C^{(MN-kq) \times kq}}.
\end{equation}
We summarize this analysis in a corollary.
\begin{corollary}
    \label{cor:pAAA}
    Consider the data \eqref{eq:Lpart} and 
    let the corresponding barycentric rational approximant $\Hr (s,p) $ have the form in \eqref{eq:Hrbary}.
    \begin{enumerate}
        \item[~~(a)] 
        If \eqref{eq:int2} holds, 
        then
        \[ 
            \Hr (\s_i,\p_j) = H (\s_i,\p_j) 
            ,
            \quad i = 1, \dots, k, 
            ~ j = 1, \dots, q
            . 
        \]
        
        \item[~~(b)]
        Assume \eqref{eq:int2} holds.
       Choose the indices $\alpha_{ij}$ using 
       \begin{equation}
       [ \alpha_{11} \cdots \alpha_{1q} ~|~ \cdots ~|~ \alpha_{k1} \cdots \alpha_{kq} ] = \ba^\star 
       \quad \mbox{where}\quad \ds
        \ba^\star = \argmin_{\| \ba \|_2 = 1} \| \Loew_2 \ba \|_2,
       \end{equation}
       where $\Loew_2$ is as defined  in \cref{eq:l2mat}, with
       $\Loewh$ is as given by 
       \cref{eq:loew2},
       $\Loewhs$ by \cref{Lsigma} and \cref{Lsigmailj}, and 
       $\Loewhp$ is defined as
        \begin{align*}
            \Loewhp 
            & = 
            \left[ \begin{array}{ccc|c|ccc}
                \Loew_{\p_1} \be_1 && & & \Loew_{\p_1} \be_k &&  \\
                & \ddots & &\cdots && \ddots & \\
                && \Loew_{\p_q} \be_1 & &&& \Loew_{\p_q} \be_k 
            \end{array} \right] {\in \C^{(q(N-k))\times(kq) }} ,
            \end{align*}
            where
            \begin{align} \label{Lpij}
            \Loew_{\pi_j} (\ih,i)
            & = 
            \frac{ H (\sh_{\ih},\p_j) - H (\s_i,\p_j) }{ \sh_{\ih} - \s_i }
            , ~~ \ih = 1, \dots, N-k , ~~ i = 1, \dots, k,
        \end{align}
        {and $\be_i\in\C^k$ is the $i$th unit vector.}
        Then, the two-variable barycentric approximant minimizes the linearized LS error
        \[
            \Hr = \argmin_{\hat{H} = n/d} \sum_{i,j} | H(s_i,p_j) d (s_i,p_j) - n (s_i,p_j) |^2
        \]
        for the samples $(s_i,p_j)$ corresponding to 
        the error $\beps$ in \cref{errorindata}, i.e., for the data in $\{ \Fhs,\Fhp,\Fh\}$.
    \end{enumerate}
\end{corollary}

\subsubsection*{Choosing the interpolated vs LS-fitted data}
The last component of \paaa is  determining how to choose the data to be interpolated and the data to be fitted in the LS sense.  Let
$\Hr(s,p)$ in \cref{eq:Hrbary} be the 
current \paaa approximant  corresponding to the interpolation/LS partitioning in \cref{eq:Lpart}. 
Note that the order of the current approximation is $(k-1,q-1)$ and 
these orders need not be equal.
Then,  we select the next frequency-parameter tuple $(\s_{k+1},\p_{q+1})$ by means of the greedy search
\begin{equation} \label{greedy2d}
    (s_{\ih},p_{\jh}) = \argmax_{(i,j)} 
    | H(s_i,p_j) - \Hr(s_i,p_j) | .
\end{equation}
We do not simply set $(\s_{k+1},\p_{q+1}) = (s_{\ih},p_{\jh})$
since one of the entries might already be in the previous interpolation data. In other words, $s_{\ih}$ might already be in the set
$\bs$ or $p_{\jh}$ might already be in the set
$\boldp$ in \cref{eq:Lpart}. {We note that} this cannot occur for $s_{\ih}$ and $p_{\jh}$ simultaneously since we impose interpolation on the selected tuples. In other words, if  the tuple  $(s_{\ih},p_{\jh})$ was already in the interpolated data, we would have had 
 $ H (s_{\ih},p_{\jh}) - \Hr (s_{\ih},p_{\jh}) = 0 $, which means the whole data set is interpolated.
 If the point $p_{\jh}$ is already in the set $\boldp$ in \cref{eq:Lpart},
 then the order in the variable-$p$ remains unchanged as $q-1$ and 
 the set $\boldp$ is not altered. On the other hand, the point 
 $s_{\ih}$ is added to set $\bs$ in \cref{eq:Lpart}
 and  the order in the variable-$s$ is increased to $k$. {Conversely, $p_{\jh}$ is added to $\boldp$ and $s_{\ih}$ is not added to $\bs$ if the point $s_{\ih}$ is already in the set $\bs$.} This allows updating the orders in each variable independently, giving the algorithm flexibility to make the decision automatically.
Once  the 
data partitioning  \cref{eq:Lpart} (and the orders) are updated, \paaa computes the new coefficients $\beta_{ij}$ as 
in \cref{eq:int2}, and then solves the LS problem \cref{eq:l2mat} for  the updated coefficient vector $\ba$. The process is repeated until either a pre-specified error tolerance or desired orders  in $(s,p)$ are achieved. We give a brief sketch of 
\paaa in \cref{alg:p-AAA}. We use the notation $[x_{ij}]$ to denote a matrix whose $(i,j)$th entry is $x_{ij}$.
    \begin{algorithm}
    \caption{\paaa}
    \label{alg:p-AAA}
        \begin{algorithmic}[1]
        \STATE{Given $ \{s_i\}$, $\{p_j\}$, and $ \{h_{ij}\} = \{H (s_i,p_j) \}$}
        \STATE{Initialize: $ k = 0 $ and $ q = 0 $}
        \STATE{Define $ \Hr = average(h_{ij})$ and set error $ \gets \frac{\| [h_{ij}] - [\Hr] \|_\infty}{\|[h_{ij}]\|_\infty} $}
        \WHILE{error $>$ desired tolerance}
        \STATE{Select $ (s_{\ih}, p_{\jh}) $ by the greedy search \cref{greedy2d}}
        \STATE{Update the data partitioning \cref{eq:Lpart}:}
        \IF{$s_{\ih}$ was not selected at a previous iteration}
            \STATE{$k \gets k+1$}
            \STATE{$\s_k \gets s_{\ih}$}
        \ENDIF
        \IF{$ p_{\jh}$ was not selected at a previous iteration}
            \STATE{$q \gets q+1$}
            \STATE{$\p_q \gets p_{\jh}$}
        \ENDIF
        \STATE{Build $\Loew_2$ as in \cref{eq:l2mat}}
        \STATE{Solve $ \min \| \Loew_2 \ba \|_2 $ s.t. $\|\ba\|_2=1$} \label{lst:line:lsq}
        \STATE{Use $\ba$ to update the rational approximant $ \Hr (s,p)$ with \eqref{eq:Hrbary}--\eqref{eq:int2}}\label{lst:line:Hr}
        \STATE{error $ \gets \frac{\| [h_{ij}] - [\Hr (s_i,p_j)] \|_\infty}{\|[h_{ij}]\|_\infty} $} \label{lst:line:convergencecriterion}
        \ENDWHILE
        \RETURN $\Hr$
    \end{algorithmic}
    \end{algorithm}

\begin{remark} \label{paramLoewremark}
\emph{Parametric Loewner framework.} As in the single-variable case discussed in \Cref{sec:loewner1d}, one can choose to construct an approximation that interpolates the full-data \cref{eq:prob} as done in \cite{Io14data,An12two}. In this case,
 based on the ranks of  Loewner matrices, the orders $k$ and $q$ are chosen large enough so that, unlike in  \paaa, the matrix $\Loew_2$ has a null space and thus one chooses the coefficient vector $\ba$ by solving the linear system $\Loew_2 \ba = \bf0$. Therefore, the parametric Loewner framework \cite{Io14data,An12two} interpolates the full data in contrast to \paaa, which  greedily chooses a subset of data to interpolate and performs LS fit on the rest. When the orders
$k$ and $q$ are not chosen \emph{large enough}, the parametric Loewner framework no longer yields an interpolant, and instead a {Loewner \emph{approximant}} is obtained. For details we refer the reader to \cite{An12two,Io14data,ALI17,AntBG20}. Even though this situation is more similar to the case of \paaa, the major difference lies in the fact that \paaa is an iterative algorithm and chooses the interpolation data with a greedy search while performing LS fit on the rest. In other words, \paaa decides the data-partitioning \cref{eq:Lpart} automatically using a greedy search with an appropriately defined criterion. On the other hand, the parametric Loewner framework is a one-step algorithm and how to partition the data is not yet fully understood. Even though there have been recent efforts in this direction for the single-variable case \cite{karachalios2018data,karachalios2017case,embree2019pseudospectra}, this is still an open question, especially in the multivariate case. It will be worthwhile to investigate how the final data partitioning from \paaa affects the parametric Loewner construction and whether it improves the conditioning-issues, appearing, at times, in the (one-step) Loewner framework. 
\end{remark}
{
\begin{remark} \emph{Real state-space realization.} When working with dynamical systems, it is often desirable to have access to system matrices that constitute a state-space form similar to the one presented in \eqref{paramlti}. The system matrices are typically real-valued, a desirable property to retain in the rational approximant as well. As outlined in \cref{appendix:realization}, real state-space representations based on two-variable barycentric forms can be computed if all samples in the $p$ and $s$-variables are real valued~\cite{Io14data}. In the dynamical system setting the parameter samples  are generically real valued whereas the frequency are usually complex-valued. In order to ensure realness in the complex case, the frequencies need to be sampled in complex-conjugate pairs. This means that if $s_i \in \C$ is sampled, we also sample $\overline{s}_i$.
Then, if $s_{\ih}$ in Step 7 of \cref{alg:p-AAA} is a complex frequency,  we also add $\overline{s}_{\ih}$ to the interpolation data set and Line 8 of \Cref{alg:p-AAA} becomes $k \gets k+2$. 
We follow this approach in the examples discussed in  \Cref{sec:sriram,sec:gyro}. Algorithmic details are explained in \cref{appendix:realization}.
\end{remark}
}
{\begin{remark}
An important property of the single-variable \aaa algorithm is that either one obtains an approximant with a desired accuracy or an interpolant of minimal order. Although \paaa has similar properties, the interpolant may not be of minimal order. (This is illustrated in the numerical example of \Cref{sec:exsyn}.) We emphasize that this is only an issue for small synthetic examples as we consider in \Cref{sec:exsyn} where the underlying model is a low-order multi-parameter rational function to begin with. In most practical situations of interest (indeed for all the other examples we have considered), we obtain an approximant; not an exact recovery. A post-processing routine which ensures minimal order of interpolants (in case they occur) is presented in \cref{appendix:mininterpolant}.
\end{remark}}

\subsection{Numerical Examples}
Next, we illustrate the performance of \paaa on three numerical examples.
\subsubsection{Synthetic Transfer Function}

\label{sec:exsyn}
We use a simple model from \cite{Io14data}, which is a low-order rational function in two variables.
Consider
    \[
	H (s,p) = \frac{ 1 }{ 1 + 25 (s+p)^2 } + 
	\frac{ 0.5 }{ 1 + 25 (s-0.5)^2 } +
	\frac{ 0.1 }{ p + 25 } .
	\]
We sample this transfer function at $ H (s_i,p_j) $ for $ N = M = 21 $ frequency and parameter points linearly spaced in $ s_i \in [-1, 1 ] $ and  $ p_j \in [0, 1] $.
This is a rational function with order $(4,3)$. 
\paaa terminates after 7 iterations.  \cref{tab:syn_aaap} shows the greedy search selection at each iteration step. {Additionally, quantities related to the post-processing step presented in \cref{appendix:mininterpolant} are shown.}
	\begin{table}[htbp]
	\centering
	\begin{tabular}{|l|c|r|l|c|c|c|c|}
	\hline
	iter. & greedy selection & $\s_k$ & $\p_q$ & $(k,q)$ & $\dim \ker \Loew_2$ \\ \hline
	1 & $(0,0)$ & 0 & 0 & (1,1) & 0 \\ 
	2 & $(-1,0)$ & -1 & & (2,1) & 0 \\
	3 & $(0.1,0)$ & 0.1 & & (3,1) & 0 \\
	4 & $(0,1)$ & & 1 & (3,2) & 0 \\
	5 & $(-1,0.6)$ & & 0.6 & (3,3) & 0 \\
	6 & $(-0.6,0.1)$ & -0.6 & 0.1 & (4,4) & 0 \\
	7 & $(0.6,0.55)$ & 0.6 & 0.55 & (5,5) & 2 \\  \hline
	\multicolumn{4}{|c|}{post-processing as in \Cref{appendix:mininterpolant}} & (5,4) & 1 \\
	\hline
	\end{tabular}
	\caption{Example \ref{sec:exsyn} \paaa samples selected at each iteration}
	\label{tab:syn_aaap}
	\end{table}
Note that the \paaa approximation $\Hr$ {(without the post-processing) would have been of order
$(k-1,q-1) = (4,4)$}, as opposed to $(4,3)$ of the original model. 
This is due to the greedy search selecting frequencies and parameters to interpolate as tuples hence allowing for repetition.
In \cref{tab:syn_aaap} we see exactly how this happened for this example.
During iterations 2 and 3, no parameters are added for interpolation while during iterations 4 and 5, no frequencies are added for interpolation. 
Upon convergence, for this simple example where the underlying function is a low-order rational function itself, \paaa exactly recovers it. In other words, after step 7, all the data is interpolated. This shows another flexibility of \paaa. If the underlying order is low enough, the LS component is automatically converted to a full interpolation, thus, in this special example, giving the same approximant as the parametric Loewner approach \cite{Io14data}.

We present in \Cref{fig:syn_aaap} the evolution of the \paaa approximant at various iterations: first, third, and last (seventh).
As \Cref{fig:syn_aaap} shows that, upon convergence, the proposed algorithm captures the full model exactly.
	\begin{figure}[htbp]
	\centering
	\includegraphics[width=1\textwidth]{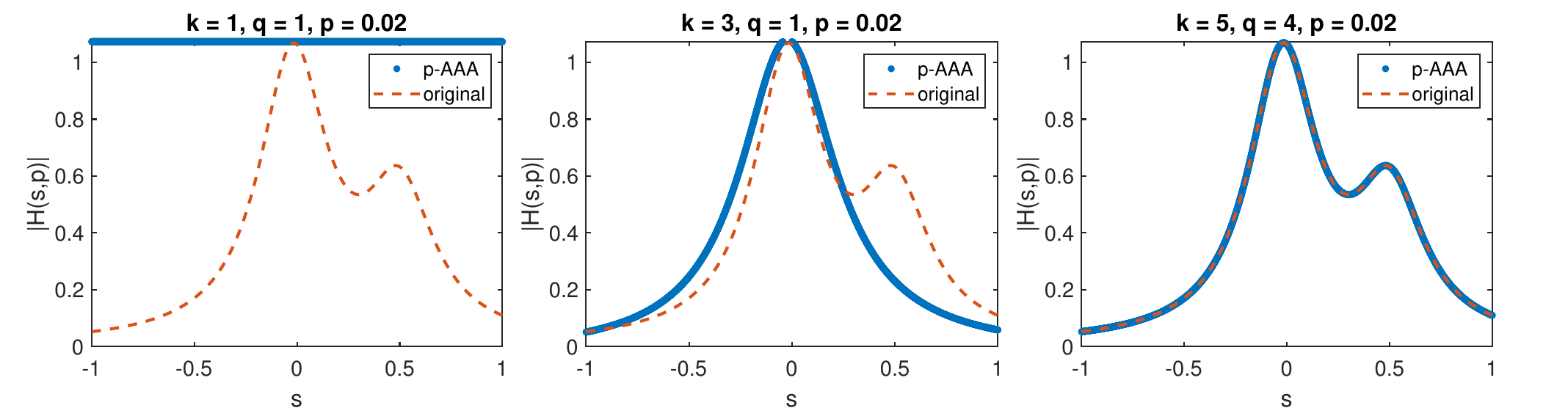}
	\caption{Example \ref{sec:exsyn}: \paaa approximation at various iterations}
	\label{fig:syn_aaap}
	\end{figure}
\subsubsection{A beam model}
\label{sec:sriram}
In this example, we consider the finite element model of a one-dimensional Euler-Bernoulli beam 
 with a string attached near its left boundary and  an input force  applied at its right boundary, as shown in
\Cref{fig:beam}. 
As \cool{for} the output $y(t)$, we measure the displacement at the right boundary where the forcing is applied.
    \begin{figure}
        \centering
        \includegraphics[width=\textwidth]{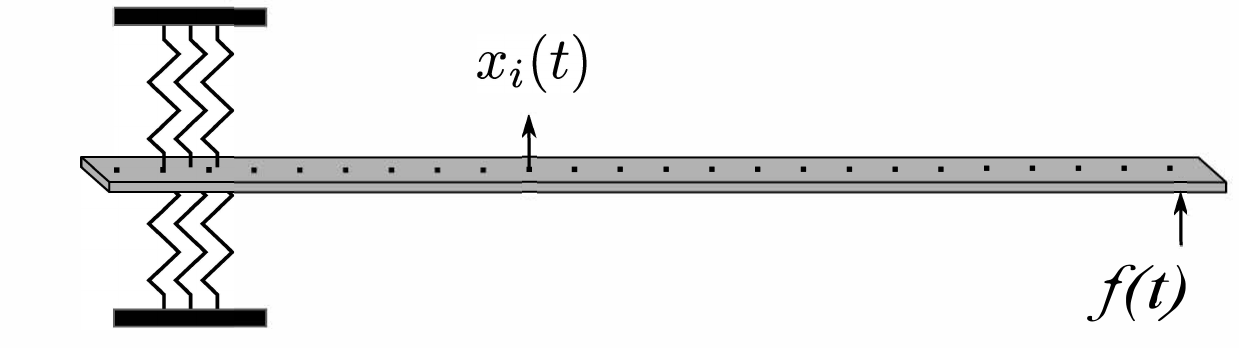}
        \caption{Example \ref{sec:sriram}: Visualization of an Euler-Bernoulli beam}
        \label{fig:beam}
    \end{figure}
We take the stiffness coefficient of the spring as the parameter and obtain  the parametric dynamical system
$$
\bM \ddot{\bx}(t,p) + \bG \dot{\bx}(t,p) + \bK(p) \bx(t,p) = \bb \oldu(t),~~~y(t,p) = \bc^\top \bx(t,p),
$$
with the corresponding transfer function 
$$H(s,p) = \bc^\top(s^2 \bM + s \bG + \bK(p))^{-1}\bb,$$
where $\bM$ and $\bG$ are, respectively, the mass and damping matrices; $\bK(p)$ is the parametric stiffness matrix; and $\bb$ and $\bc$ are, respectively, the input-to-state  and the state-to-output mappings.
We measure the transfer function  at
$H(s_i,p_j)$
for 
$N=3000$ frequency points $\{s_i\}$ in 
the interval $[0,2\pi\times 10^3]\imath$ where $\imath^2 = -1$ and  
for $M=3$ parameter values 
$ p_1 = 0.2$, $p_2 = 0.4$, and $p_3 =1 $.
\paaa  yields an approximant with
orders $(k,q)={(19,2)}$. Out of three parameter samples, 
\paaa chooses {$p_2= 0.4$ and $p_3= 1$} for interpolation. Using the same parameter and frequency samples, we also construct the parametric Loewner approximant \cite{Io14data}.
\Cref{fig:sriramLerr} shows the amplitude frequency responses of the original transfer function $H(s,p)$, and the \paaa and parametric Loewner approximants  for various parameter values, including  values that did not enter into \paaa or parametric Loewner construction ($p=0.8$ and $p=15$ in  \Cref{fig:sriramLerr}).
Both \paaa and parametric Loewner  yield
 highly accurate approximations, capturing the peaks in the frequency response accurately. To check the accuracy of the \paaa and parametric Loewner approximants further, we perform an exhaustive search over the parameter domain by computing, for {50} linearly spaced $\hat{p} \in [0,1]$,  the worst-case frequency domain error, i.e., $\max_s \mid H(s,\hat{p})- \Hr(s,\hat{p})\mid$ where 
$s= \imath \omega$ with $\omega \in [0,2\pi\times 10^3]$. {We use $3000$ $\omega$ samples to approximate the maximum error.}
The results in  \Cref{fig:sriramL} 
show that  \paaa is accurate throughout the full parameter domain and, for this example, outperforms the parametric Loewner approach.

\begin{figure}
    \centering
    \includegraphics[width=0.5\textwidth]{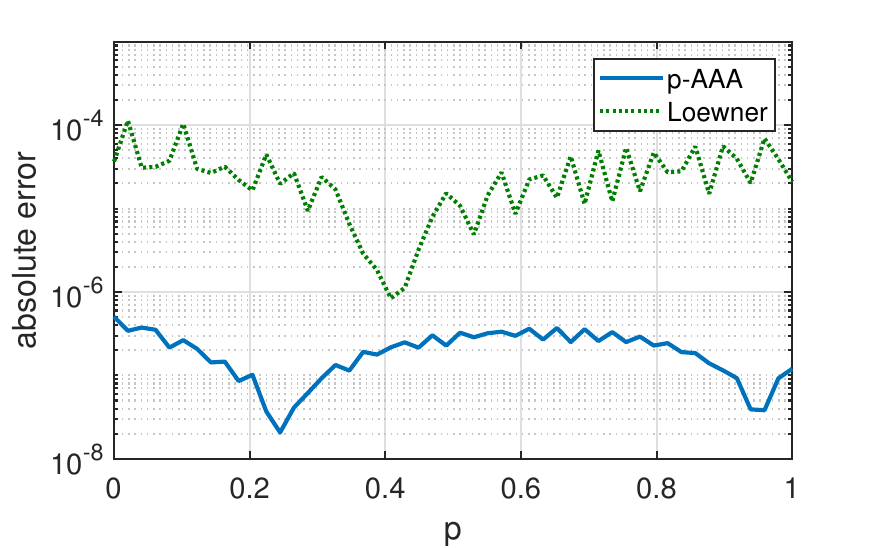}
    \caption{Example \ref{sec:sriram}:
    \paaa approximation for various parameter values and
    Loewner approximation with the same order as \paaa.
    }
    \label{fig:sriramL}
\end{figure}

\begin{figure}
    \centering
    \includegraphics[width=1\textwidth]{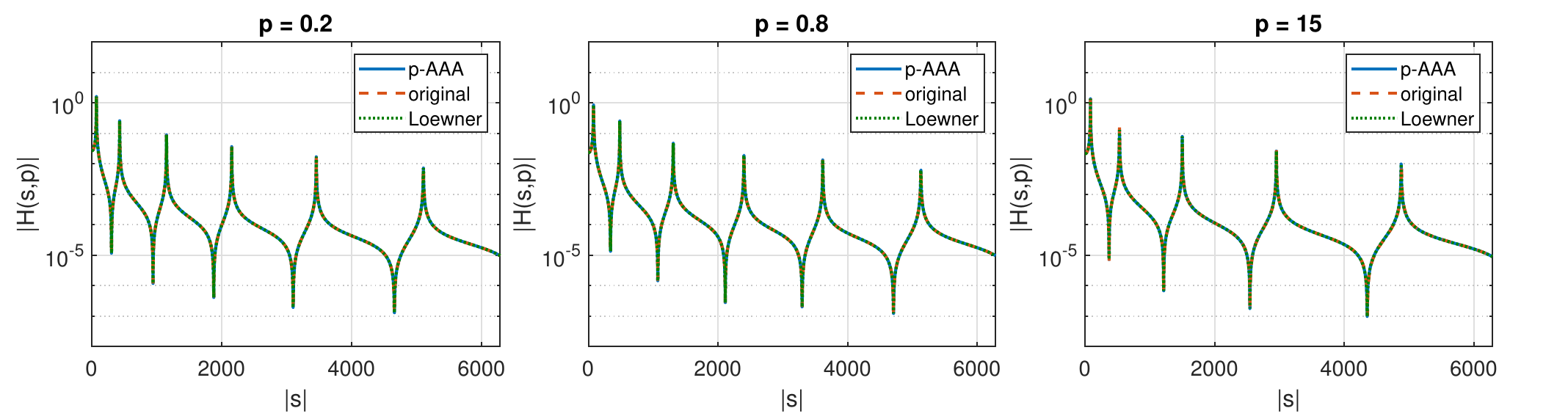}
    \caption{{Example \ref{sec:sriram}: \paaa and Loewner approximations in the s-interval sampled.}}
    \label{fig:sriramLerr}
\end{figure}

{
\subsubsection{\paaa convergence behaviour}
\label{ex:convergence}
\begin{figure}[htbp]
\label{fig:convergence}
\includegraphics[width=0.45\textwidth]{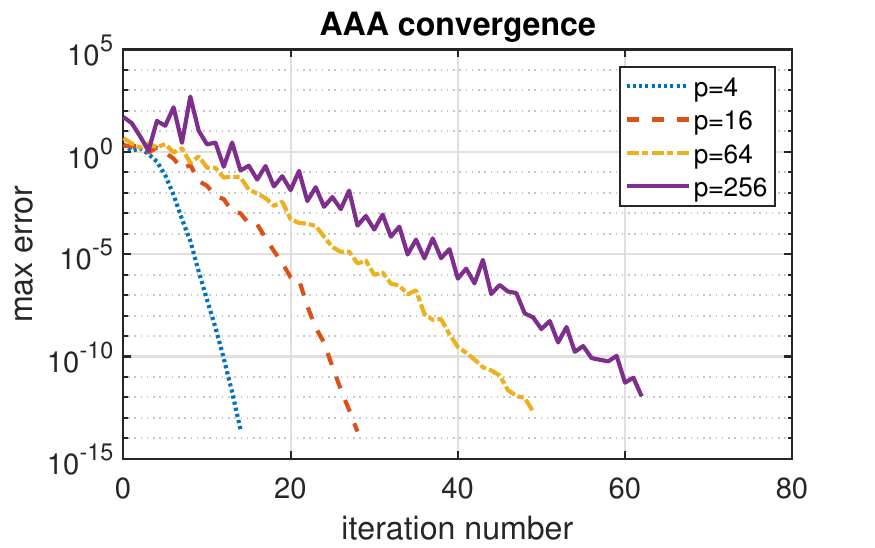}
\includegraphics[width=0.45\textwidth]{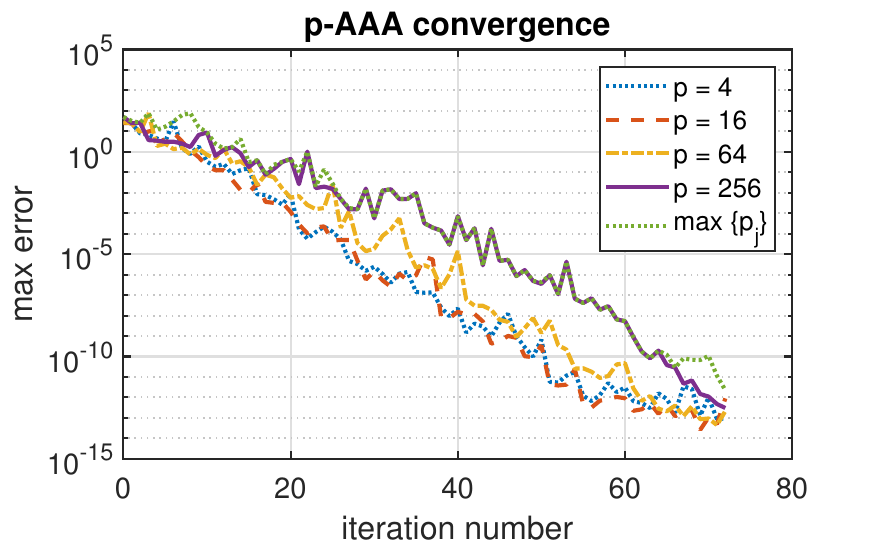}
\caption{Example \ref{ex:convergence}: Convergence of \aaa compared with \paaa.}
\end{figure}
In this section we demonstrate the convergence behavior of \paaa using a general multivariate function, not related to dynamical systems. To do so we consider an example from \cite{Na18AAA} where the goal is to approximate the function $\tan(ps)$. We take $N=1000$ equispaced sample points on the unit circle for the $s$ variable and 
a set of $9$ parameter samples $\{2^0,2^1,\ldots,2^8\}$.
For a comparison, (the single variable) \aaa has been executed for $p=4,16,64,256$ 
We note that \aaa has been run for every $p$ value separately. This is in contrast to \paaa where \paaa is run only once and the resulting parametric approximant can be used for any given parameter value. 
For both algorithms a relative error tolerance of $10^{-13}$ was used. The parametric rational approximant computed by \paaa after $73$ iterations is of order $(70,8)$.
In \Cref{fig:convergence} we illustrate the differences in the convergence behaviour of \aaa and \paaa implementations. The left-hand side plot in \Cref{fig:convergence} depicts maximum errors over all sampled $s$-variables during individual (single variable) \aaa runs for the four parameter choices of  $p=4,16,64,256$.  The right-hand side plot in \Cref{fig:convergence} shows  the maximum error  with respect to all sampled $s$ and $p$ values, denoted by the legend ``$\max\{p_j\}$"
(corresponding to the error in Line~\ref{lst:line:convergencecriterion} of \Cref{alg:p-AAA} used as a convergence criterion).  During the \paaa implementation, we also monitor the maximum $s$-errors corresponding to the  $p=4,16,64,256$ samples. We emphasize that these errors values for specific $p$ values are not part of the \paaa stopping criterion. \paaa only monitors the maximum error over all the $s$ and $p$ samples. These are computed here only for comparison purposes.  \Cref{fig:convergence} illustrates that \aaa convergence speed varies with the magnitude of $p$ (faster convergence for the smaller $p$ values) whereas in \paaa errors decrease uniformly across the parameter set mainly dictated by the hardest case. Overall, \paaa needs more iterations to converge than \aaa for a given fixed parameter. However, as mentioned above, we  need to run the parametric algorithm only once in order to obtain a single approximating function for all four rational functions computed by individual \aaa runs. This example demonstrates that \paaa is a viable choice in the general multivariate rational approximation setting and by no means restricted to the approximation of system dynamics in the frequency domain.
}

\subsection{\paaa for more than two parameters}
\label{sec:multiparam}
The \paaa algorithm extends analogously to the cases with more than two variables. To keep the discussion concise, we briefly highlight the three-variable case.

In this case,  the underlying (transfer) function to approximate, $H(s,p,\phat)$, is a function of the three variables, $s,p,$ and $z$, and we 
assume access to the sampling data 
\begin{equation}
    \label{eq:prob3d}
    h_{ij\ell} = H (s_i,p_j,\phat_\ell) \in \C
   ~~\mbox{for}~~ i = 1, \dots, N,
   ~~j = 1, \dots, M,
   ~~\mbox{and}
   ~~\ell = 1, \dots, O.
\end{equation}
The approximant $\Hr(s,p,\phat)$ is represented in the 
barycentric form given by
\begin{equation} \label{Hrspz}
    \widetilde{H} (s,p,\phat) =
	\sum_{i=1}^k \sum_{j=1}^q \sum_{\ell=1}^o
	\frac{ \beta_{ij\ell} }{ (s-\s_i) (p-\p_j)(\phat-\z_\ell)}
	\! \biggn/ \!
	\sum_{i=1}^k \sum_{j=1}^q \sum_{\ell=1}^o
	\frac{ \alpha_{ij\ell} }{ (s-\s_i) (p-\p_j) (\phat-\z_\ell)},
\end{equation}
where $\{ \s_i \}$,
	$\{ \p_j \}$, and $\{ \z_\ell\}$
	are to-be-determined sampling points, subsets of
	$\{s_i\}$, $\{p_j\}$, and $\{\phat_\ell\}$, respectively. As in the 
	two-variable case, $\beta_{ij\ell}$ will be chosen 
	to enforce interpolation in a subset of the data 
	and  $\alpha_{ij\ell}$ to minimize a linearized LS error in the remaining data.

In accordance with the data \cref{eq:prob3d} and the approximant
$\Hr(s,p,z)$, partition the sampling points: 
\begin{align}
        \begin{split}
        \label{eq:3ddata}
        [s_1, \dots, s_N] & = [ \s_1, \dots, \s_k ] \cup 
		    [ \sh_1, \dots, \sh_{N-k} ]  = 
		    [\bs~|~\bsh],
		    \\
      [p_1, \dots, p_M] & = [\p_1, \dots,\p_q ] \cup 
            [\ph_1, \dots, \ph_{M-q}] 
             = 
		    [\boldp~|~\bph],~\mbox{and} \\
      [\phat_1, \dots, \phat_O] & = [ \z_1, \dots, \z_o]~ \cup~ 
            [\zh_1, \dots, \zh_{O-o}] 
             = 
		    [\bzeta~|~\bzetah]. \\
        \end{split}
    \end{align}	

Then, \paaa imposes interpolation on the samples $\{\bs,\boldp,\bzeta\}$ by setting
\begin{equation} \label{betaijl}
    \beta_{ij\ell} = H ( \s_i, \p_j, \z_\ell ) \alpha_{ij\ell},~~
    \mbox{for}~i=1,\ldots,k,~j=1,\ldots,q,~\mbox{and},~\ell=1,\ldots,o.
\end{equation}
Based on the partitioning \cref{eq:3ddata}, consider the data as a three-dimensional tensor. We enforce interpolation in the $(1,1,1)$ block of this tensor with the choice in \cref{betaijl}. Then, \paaa minimizes the linearized LS error in the rest of the data by choosing the remaining coefficients
$\ba = [\alpha_{111} \cdots \alpha_{11o} | \alpha_{121} \cdots \alpha_{12o} | \cdots | \alpha_{kq1} \cdots \alpha_{kqo} ]^\top$
via the linear  LS problem 
${\displaystyle \min_{\|\ba\|_2=1} \| \Loew_3 \ba \|_2}$
where $\Loew_3$ is the 3D Loewner matrix, which plays the same role the 2D Loewner matrix $\Loew_2$ played in \Cref{sec:aaaparam}. Partioning of  the data in \cref{eq:3ddata} is 
automatically established via the greedy search in every step. 

Generalization to functions of more than three variables follows analogously. We skip those details due to cumbersome notation. 
However the potential computational difficulties with the increasing number of variables is worth elaborating. Assume that at the current step of \paaa, we have the approximant $\Hr(s,p,z)$ as in \cref{Hrspz}.
Given the sampling data in \cref{eq:prob3d}, this will result in $\Loew_3$ having $NMO-kqo$ rows and $kqo$ columns. 
Therefore computing the coefficient vector $\ba$ becomes more expensive as the number of variables (and the orders in each variable) increase. For functions with many variables, if the coefficient matrix becomes prohibitively  large to compute $\ba$ via direct methods, one might revert to well-established iterative approaches. For the numerical examples we considered in this paper, these computational complications did not arise and direct methods were readily available to apply.

{
\subsubsection{Parameterized Gyroscope Model}
\label{sec:gyro}
In this section, we use \paaa to approximate the dynamics of a microelectromechanical system (MEMS) gyroscope. The benchmark is  available through \cite{morwiki_modgyro} and further information regarding the background as well as the the operation principle of the MEMS gyroscope are discussed in \cite{morMoo07}. Similar to the example in \Cref{sec:sriram}, the time-domain description of the system is given by the second-order model
$$
\bM(p) \ddot{\bx}(t,p,z) + \bG(p,z) \dot{\bx}(t,p,z) + \bK(p) \bx(t,p,z) = \bb,~~~y(t,p,z) = \bc^\top \bx(t,p,z),
$$
where the mass matrix $\bM(p) = \bM_1 + p\bM_2$, damping matrix $\bG(p,z)=z(\bG_1+p\bG_2)$ and stiffness matrix $\bK(p) = \bK_1 + \frac{1}{p} \bK_2 + p\bK_3$ are defined with respect to the structural parameter $p$ and the rotation velocity $z$. We use \paaa to approximate the corresponding three-variable transfer function
$$H(s,p,z) = \bc^\top(s^2 \bM(p) + s \bG(p,z) + \bK(p))^{-1}\bb,$$
in the operating frequency range of the device, which corresponds to $s\in [2\pi \times 0.025, 2\pi \times 0.25 ]\imath$. For this example we chose to sample $100$ linearly spaced frequencies in the aforementioned interval as well as $10$ linearly spaced points in $[1,2]$ for the $p$ parameter and $10$ logarithmically spaced points in $[10^{-7},10^{-5}]$ for the $z$ parameter. After $33$ iterations of \paaa we obtain an approximant with order $(k,q,o) = (62,7,9)$ and a maximum relative error of $8.9 \times 10^{-4}$ throughout the sampled domain. 
\Cref{fig:gyro} depicts the transfer function  $H(s,p,z)$ for multiple unsampled parameter values. The frequency response drastically varies for different parameters, thus making it a function which is difficult to approximate. This may partially be due to the non-linear parameter dependence of the matrix $\bK(p)$. In spite of these difficulties, \paaa is able to produce good approximations for most parameters in the intervals of interest.
\begin{figure}[htbp]
\includegraphics[width=1\textwidth]{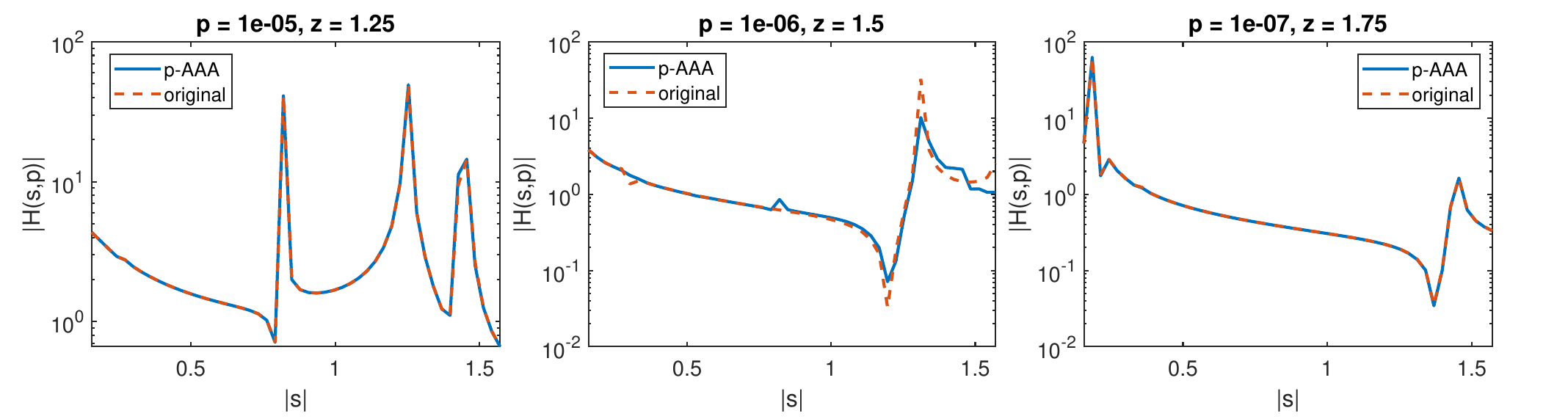} 
\caption{Example \ref{sec:gyro}: \paaa approximation of gyroscope model for various parameter combinations.}
\label{fig:gyro}
\end{figure}
}
\section{\paaa for matrix-valued functions}

\label{subsec:mimoAAAp}

So far, we have considered approximating scalar-valued functions $H(s,p)$. In this section, we  discuss \paaa for approximating matrix-valued functions instead. This is a common situation, especially arising in the case of dynamical systems where the underlying system has multiple-inputs and multiple-outputs (\textsf{MIMO}), leading to matrix-valued transfer functions.
Motivated by our interest in approximating dynamical systems,  we will call the resulting method \mvpaaa. To keep the notation concise, we will present the discussion for the two-variable case. But as in
\Cref{sec:multiparam}, the results similarly extend to higher-dimensional parametric problems. 

Let $\bH(s,p)$ denote the underlying
\textsf{MIMO} (transfer) function with $\nin$ inputs and $\nout$ outputs. Therefore, for the sampling points $\{s_i\}_{i=1}^N$
and $\{p_j\}_{j=1}^M$, we
have access to the matrix-valued sampling data:
\begin{equation} \label{mimodata}
    \bH_{ij} = \bH (s_i,p_j) \in \C^{\nin \times \nout} ~~
    \mbox{for}~
    i = 1,\dots,N ~~\mbox{and}~~
    j = 1, \dots, M.
\end{equation}
From the data 
\cref{mimodata},
the goal is to construct a high-fidelity, matrix-valued approximant $\bHr(s,p)$ to $\bH(s,p)$. 

\subsection{Transformation to  scalar-valued data}

For the single-variable (nonparametric case), 
one solution to handle the matrix-valued data in \aaa is to vectorize every sample and replace the scalar data forming the Loewner matrix $\Loew$ with the vectorized data. This is closely related to the approach 
proposed in
Lietaert \emph{et al.} \cite{lie18auto} for using \aaa in nonlinear eigenvalue problems. It is also analogous to how \vf handles \textsf{MIMO} problems. One potential disadvantage of this approach is that, in the case of large number of inputs and outputs, the resulting  Loewner matrix will have large dimensions, leading to a computational expensive LS step. Exploiting the fact that 
only certain rows and columns of the underlying Loewner matrix change in every step,  \cite{lie18auto} partially 
alleviates this computational complexity. However, for the parametric problems we consider here,  dimension growth due to vectorization is more prominent and we will adopt another approach introduced by \cite{El17conversions} for the nonparametric case, which transforms the \textsf{MIMO} data to a scalar one, and apply \aaa to this scalar-valued data.  We will extend this approach to  parametric problems and establish what it means, for \mvpaaa,   in terms of  interpolation and the LS minimization.

As in the scalar case, assume the  partitioning of the  data in
\cref{mimodata} as follows:
\begin{align}
        \begin{split}
        \label{mimodatapart}
        \{s_1, \dots, s_N\} & = \{ \s_1, \dots, \s_k \} \cup 
		    \{ \sh_1, \dots, \sh_{N-k} \}  \eqdef 
		    \{\bs~\cup~\bsh\},
		    \\
      \{p_1, \dots, p_M\} & = \{\p_1,\dots,\p_q\} \cup 
            \{\ph_1, \dots, \ph_{M-q}\} 
             \eqdef 
		    \{\boldp~\cup~\bph\},~\mbox{and} \\
      &
      \renewcommand{\arraystretch}{1.2} \left[ \begin{array}{c|c}
		[ \bH (\s_i,\p_j) ] & [ \bH (\s_i,\ph_j) ] \\\hline
		[ \bH (\sh_i,\p_j) ] & [ \bH (\sh_i,\ph_j) ]
		\end{array} \right] \renewcommand{\arraystretch}{1}
		\eqdef
		\renewcommand{\arraystretch}{1.3}
		\Dpart.
		\renewcommand{\arraystretch}{1}
        \end{split}
    \end{align}	
    This partitioning will be determined by applying \paaa to a scalar data set described below.
In accordance with this partitioning, we want to construct $\bHr (s,p)$ with the matrix-valued barycentric form
\begin{equation}
    \label{eq:Hmimo}
    \bHr (s,p) =
    \frac{\bN (s,p) }{ d (s,p) } = 
    \sum_{i=1}^k \sum_{j=1}^q 
	\frac{ \bB_{ij} }{ (s-\s_i) (p-\p_j) }
	\biggn/
	\sum_{i=1}^k \sum_{j=1}^q 
	\frac{ \tilde{\alpha}_{ij} }{ (s-\s_i) (p-\p_j) },
\end{equation}
where $ \bB_{ij} \in \C^{\nin \times \nout} $ and $\tilde{\alpha}_{ij} \in \C$ are to be determined. 

Motivated by \cite{El17conversions} for the nonparametric case, we convert the matrix-valued data \cref{mimodata} to the scalar one by picking two random unit vectors $ \bw \in \C^{\nout}$ and $\bv \in \C^{\nin}$, and computing 
\begin{equation}
    \label{eq:datamimo}
   h_{ij} = \bw^\top \bH(s_i,p_j) \bv ~~~\mbox{for}~~
    i = 1,\dots,N ~~\mbox{and}~~
    j = 1,\dots,M.
\end{equation}
We apply \paaa to the scalar data \cref{eq:datamimo} to obtain the 
scalar-valued rational approximation,
as in \cref{eq:Hrbary}:
\begin{equation}
\label{mvpaaaHr}
    \Hr (s,p) = \frac{n(s,p)}{d(s,p)} = 
    \sum_{i=1}^k \sum_{j=1}^q 
	\frac{ \left( \bw^\top \bH (\s_i,\p_j) \bv \right) \alpha_{ij} }{ (s-\s_i) (p-\p_j) }
	\biggn/
	\sum_{i=1}^k \sum_{j=1}^q 
	\frac{ \alpha_{ij} }{ (s-\s_i) (p-\p_j) }.
\end{equation}
Note that $\beta_{ij} = \bw^\top \bH (\s_i,\p_j) \bv \alpha_{ij} $. Then, the final matrix-valued approximant $\bHr(s,p)$
 is obtained by setting
 $\tilde{\alpha}_{ij} = \alpha_{ij}$
 and $ \bB_{ij} = \alpha_{ij} \bH (\s_i,\p_j) $  in \cref{eq:Hmimo}, resulting in 
 \begin{equation}
    \label{eq:Hmimofinal}
    \bHr (s,p) =
    \frac{\bN (s,p) }{ d (s,p) } = 
    \sum_{i=1}^k \sum_{j=1}^q 
	\frac{ \bH_{ij} \alpha_{ij}}{ (s-\s_i) (p-\p_j) }
	\biggn/
	\sum_{i=1}^k \sum_{j=1}^q 
	\frac{ {\alpha}_{ij} }{ (s-\s_i) (p-\p_j) }.
\end{equation}
As in the scalar \paaa case, by construction, 
our choice of $\bB_{ij}$  guarantees interpolation of the data for the samples $\{\bs,\boldp\}$ in 
\cref{mimodatapart}.
However, the (linearized) LS minimization is different. We  summarize these results next.
\begin{prop} \label{mimopaaaresult}
Given the sampling data
 \cref{mimodata}, let 
$\bHr(s,p)$ in \cref{eq:Hmimofinal} be the resulting approximant obtained via \mvpaaa with 
$\alpha_{ij} \neq 0$ and with
the corresponding data partitioning 
\cref{mimodatapart}.
Then, $\bHr (s,p)$  interpolates the data in
 $\F$ corresponding to the samples
$\{\bs,\boldp\}$, i.e.,
    \begin{equation} 
    \label{mimointerpolation}
    \bHr (\s_i,\p_j) = 
\bH (\s_i,\p_j)~~\mbox{for}~~i=1,\ldots,k~~\mbox{and}~~j=1,\ldots,q. \end{equation}
Furthermore, $\bHr (s,p)$ 
minimizes an 
input/output weighted linearized LS measure, namely
\begin{equation}  \label{weightedls}
    \bHr = 
    \argmin_{\hat{\bH} = \bN/d} 
     \sum_{i,j} \left| 
    \bw^\top \big( 
    \bH(s_i,p_j) d (s_i,p_j) - \bN (s_i,p_j) \big) \bv \right|^2
\end{equation}
for the data in $\{\Fhs,\Fhp,\Fh\}$, not selected by the greedy search, i.e., corresponding to the sampling pairs 
   $(s_i,p_j)\in \big\{
   \{\bsh,\boldp\} \cup
    \{\bs,\bph\} \cup
   \{\bsh,\bph\}
   \big\}$.
\end{prop}
\begin{proof}
Interpolation property 
\cref{mimointerpolation} follows analogous to the scalar case, by observing that
for $\alpha_{ij} \neq 0$, $\bHr(s,p)$ has a removable pole at each $(\s_i,\p_j)$ with
    \[ \bHr (\s_i,\p_j) = 
    \frac{ \bB_{ij} }{ \alpha_{ij} } . \]
    Then, the choice $\bB_{ij} = \alpha_{ij} \bH_{ij}$ proves
    \cref{mimointerpolation}.
    
To prove \cref{weightedls}, first recall that 
$\Hr(s,p)$ in \cref{eq:datamimo} is obtained by applying (scalar-valued) \paaa to the data \cref{eq:datamimo}. 
Therefore, by \cref{cor:pAAA},
\begin{equation} \label{lsmiddle}
    \Hr = 
    \argmin_{\hat{H} = d/n} 
    \sum_{i,j} \mid \bw^\top \bH(s_i,p_j) \bv d (s_i,p_j) - n (s_i,p_j) \mid^2.
\end{equation}
Using
\cref{mvpaaaHr} and \cref{eq:Hmimofinal}, we have
$
    \Hr (s,p) = \frac{n(s,p)}{d(s,p)}= \bw^\top \bHr (s,p) \bv 
    = \frac{\bw^\top \bN(s,p) \bv}{d(s,p)}.
$
Therefore, 
\[
    \bw^\top \bH (s_i,p_j)\bv d (s_i,p_j) - n (s_i,p_j)
    =
    \bw^\top \big( 
    \bH (s_i,p_j) d (s_i,p_j) - \bN (s_i,p_j) 
    \big) \bv.
\]
Inserting  this last equality into \cref{lsmiddle} proves 
\cref{weightedls}.
~
\end{proof}
\begin{remark}
\Cref{mimopaaaresult} states that  for \mvpaaa,  interpolation holds analogously to the scalar case.
However, the LS minimization differs from the scalar case in that what is minimized is a weighted LS measure. More precisely, in terms of the LS aspect of \mvpaaa, the linearization is performed 
on the weighted error $\bw^\top(\bH(s,p)-\bHr(s,p))\bv$.
\end{remark}
\begin{remark} 
When the internal description of the underlying (transfer) function is available, as in \cref{paramlti} and \cref{Hinss}, projection-based approaches are commonly used to construct interpolatory parametric approximants
\cite{baur11interpolatory,benner15survey,AntBG20}. 
In this setting, for \textsf{MIMO} systems, one usually does not enforce full matrix interpolation. 
Instead, interpolation is enforced along selected \emph{tangential directions}. In other words, one picks vectors $\bw_i \in \C^{\nout}$ and $\bv_i \in \C^{\nin}$ such that 
$
\bH(\s_i,\p_j) \bv_i = \bHr(\s_i,\p_j)\bv_i
$
and/or 
$
\bw_i^\top \bH(\s_i,\p_j) = \bw_i^\top \bHr(\s_i,\p_j)
$. This is called tangential interpolation. Tangential vectors  usually vary with the sampling points. At this point, it is not clear, at least to us, how to achieve  tangential interpolation using the barycentric form \cref{eq:Hmimo}. However, inspired by this concept, 
instead of choosing two fixed vectors $\bw$ and $\bv$, 
one could pick different
vectors $\bw_i$, and $\bv_i$ for each sample $\s_i$, for example and apply \mvpaaa to the data $\bw_i^\top\bH_{ij}\bv_i$ to build the \textsf{MIMO} approximation \cref{eq:Hmimofinal} as above.
The resulting model $\bHr(s,p)$ would still interpolate the data $\F$ and minimize the LS error along varying weighted directions. In our experiments (see \Cref{ex:jiahua}),
fixed vectors $\bw$ and $\bv$ provided accurate approximations and therefore 
we do not pursue the idea of choosing different vectors here. The interpolatory parametric-Loewner approach \cite{Io14data}
 handles the vector-valued problems, i.e., $\bH(s_i,p_j) \in \C^{\nout \times 1}$, in a similar manner by choosing 
$\bw$ as vector of ones (and $\bv = 1$ since $\nin=1$). {Moreover, recently \cite{gosea2021} developed the block-AAA algorithm, which uses a generalized barycentric formula with matrix-valued weights.  Further extending that theory to parametric problems could offer different avenues to handle the parametric matrix-valued problems. Extending the framework of ~\cite{monzon2020multi} to parametric \textsf{MIMO} problems might also provide potential directions.  These issues will be investigated in future works.}

\end{remark}

\subsection{Numerical Examples: Stationary PDEs}
\label{ex:jiahua}
We consider two examples from \cite{chen19robust}. First is the following stationary PDE, briefly mentioned in \Cref{sec:intro}:
\begin{align} 
\label{eq:pde1}
    u_{xx}+p u_{yy} + \phat u 
    & = 10 \sin(8x(y-1)) 
    \quad \textup{on } \Omega = [-1,1]\times[-1,1],
\end{align}
with homogeneous Dirichlet boundary conditions.
The solution $u(x,y)$ depends on two the parameters $(p,\phat)$ and is independent of time. Therefore, the model is not a dynamical system, unlike our previous examples, yet this does not matter for our formulation {since} we simply view the solution as a function of two-variables. 
The \emph{truth model} is obtained via a spectral Chebyshev collocation approximation with 49 nodes in each direction.
We choose to approximate $u(x,y)$ on the whole domain $\Omega$; thus the output is the full solution, leading to a two-variable vector-valued function to sample $\bH(p,z) \in \R^{2401\times 1}$. For our \mvpaaa terminology, we interpret this as a model with $\nin=1$ and $\nout=2401$.
We take $N=M=10$ linearly spaced measurements 
of $\bH(p,z)$ in the parameter space $[0.1,4]\times[0,2]$. {
In \cref{eq:datamimo}, we set $\bw = \widetilde{\bw}/\lVert \widetilde{\bw} \rVert_2$ where the entries of $\widetilde{\bw}\in \R^{2401}$ result from a standard normal distribution. Also, $\bv = 1$ in this example.}
The usual projection-based approaches to PMoR would form a global basis from these samples and project the truth model into a low-dimensional space. However, we do not assume access to the truth model; but only its samples via black-box simulation, and construct our approximation directly from samples. \mvpaaa leads to an approximation with orders $q=3$ in $p$ and $o=3$ in $z$. 
To judge the quality of the approximation, 
we perform a parameter sweep in the full parameter domain and find the worst case scenario in terms of the maximum  error between  the truth model and the \mvpaaa approximation over $\Omega$. The worst-case approximation occurs for $p=1.7545$ and $z=2$, with an error of
$3.11\times10^{-2}$, showing that the \mvpaaa approximant is accurate even in the worst-case. 
This 
worst case scenario is depicted in 
 the \emph{left-pane} of \Cref{fig:jiahuaWorst} where the top-plot shows the truth model, 
the middle one the \mvpaaa approximation, and the bottom one the error plot. As the figure illustrates, 
\mvpaaa is able to recover the solution on the whole domain accurately.

We also apply \mvpaaa to a slightly revised PDE from \cite{chen19robust}:
\begin{align} \label{eq:pde2}
    (1+px)u_{xx} + (1+\phat y)u_{yy} = e^{4xy} \qquad \textup{on } \Omega = [-1,1]\times[-1,1].
\end{align}
The set-up is the same as above: 
    Dirichlet boundary conditions and
    the \emph{truth model} obtained via Chebyshev collocation, with 
    49 nodes in each direction,
leading to a two-variable vector-valued function to sample $\bH(p,z) \in \R^{2401\times 1}$.
We sample $\bH(p,z)$ at $N=M=10$ linearly spaced points in the parameter domain $ (p,\phat) \in [-0.99,0.99]\times[-0.99,0.99] $ and apply \mvpaaa. {We set
$\bw = \widetilde{\bw}/\lVert \widetilde{\bw} \rVert_2$ where 
 the entries of $\widetilde{\bw}$ result from a uniform random distribution.} As stated in \cite{chen19robust}, 
this problem is harder to approximate than the first one due to near singularities at the corners of the parameter domain.
This is automatically reflected in the approximation orders 
\mvpaaa chooses: {$q=5$} in $p$ and {$o=6$} in $z$.
As for the first PDE, we perform a parameter sweep in the full parameter domain to find the worst-case performance. In this case, the worst approximation occurs for $p=0.95$ and $z=0.99$, with an error of
$7.28\times10^{-2}$, an accurate approximation even in the worst case. We show the results from this 
worst case in
 the \emph{right-pane} of \Cref{fig:jiahuaWorst} where the top-plot shows the truth model, 
the middle one the \mvpaaa approximation, and the bottom one the error plot. 
As in the previous case, 
\mvpaaa accurately captures the full solution.

\begin{figure}[t]
    \centering
     \includegraphics[width=0.45\textwidth]{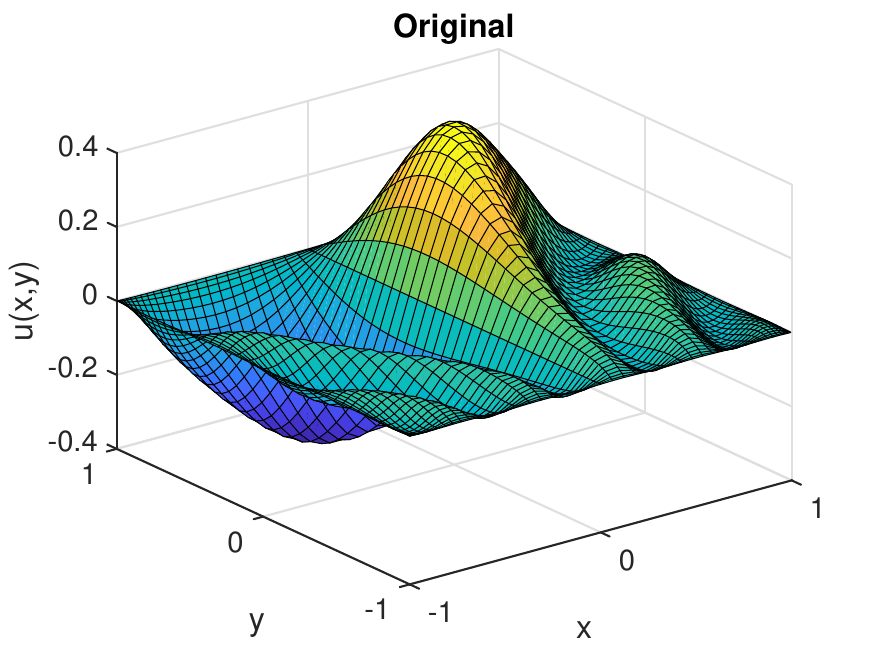}
     \includegraphics[width=0.45\textwidth]{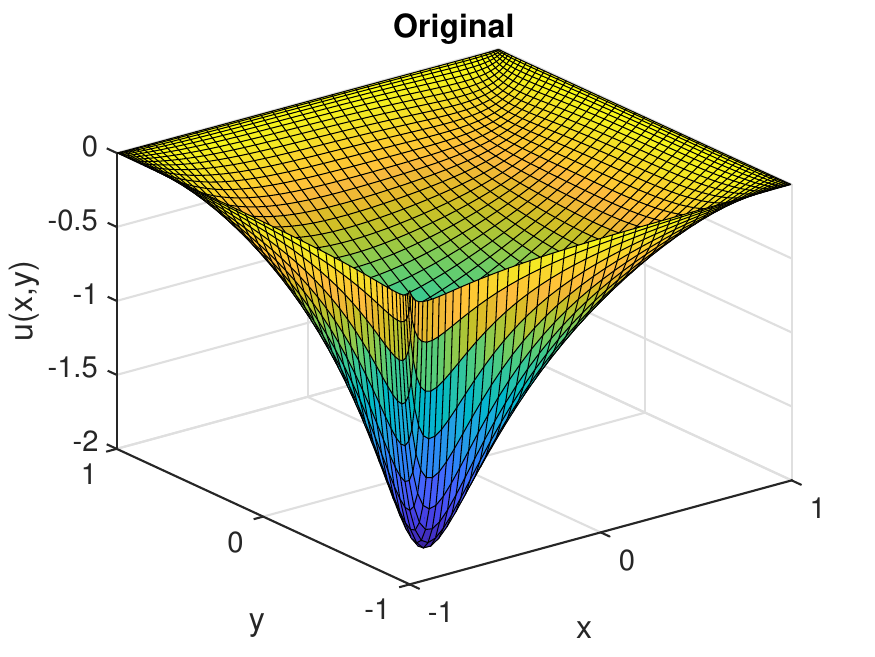} \\
     \includegraphics[width=0.45\textwidth]{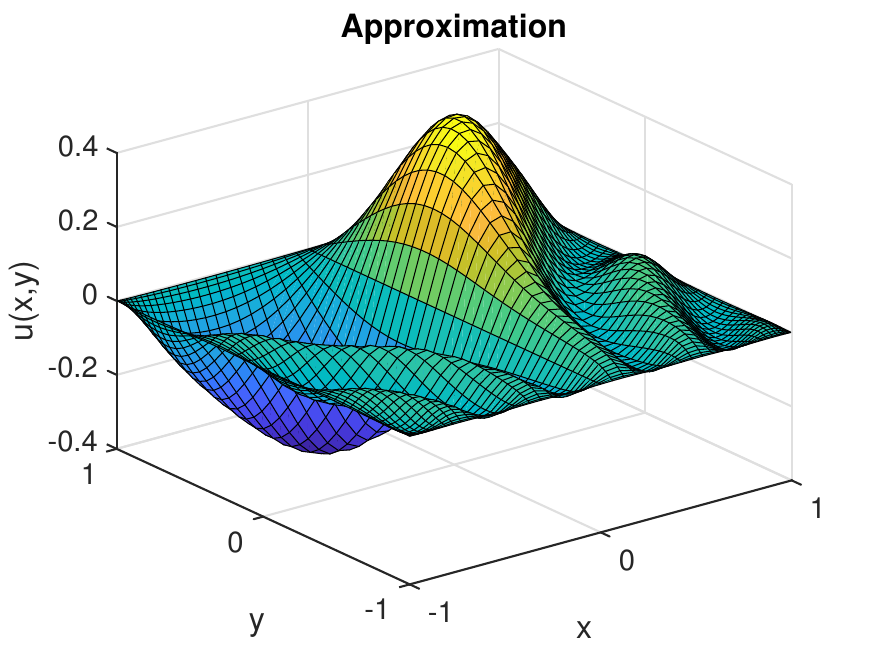}
     \includegraphics[width=0.45\textwidth]{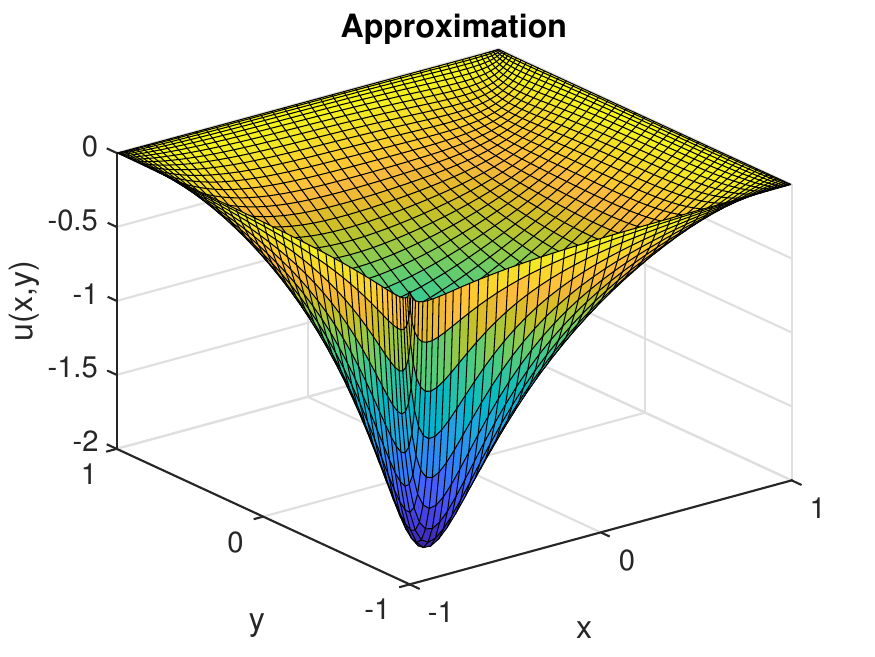} \\
     \includegraphics[width=0.45\textwidth]{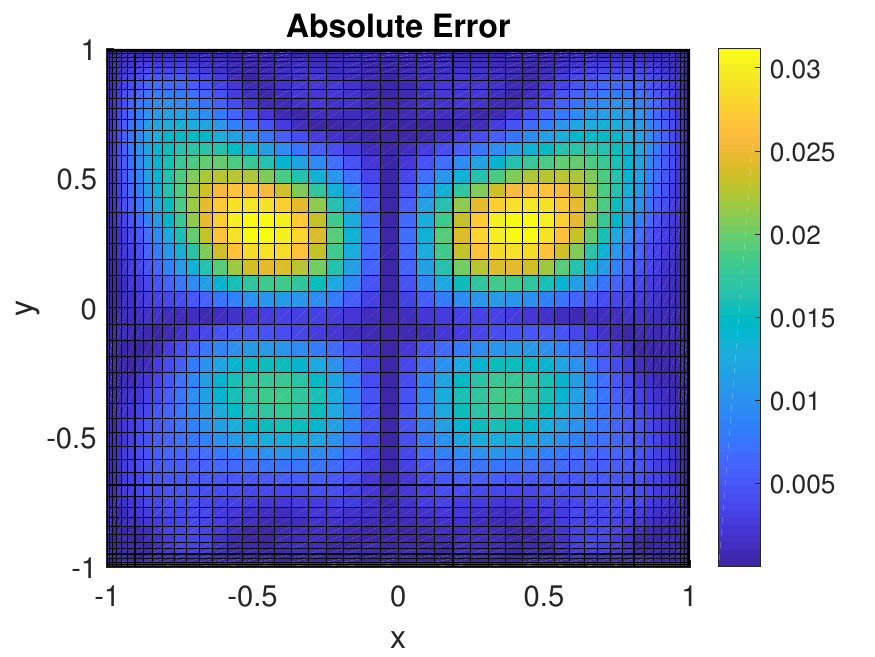}
     \includegraphics[width=0.45\textwidth]{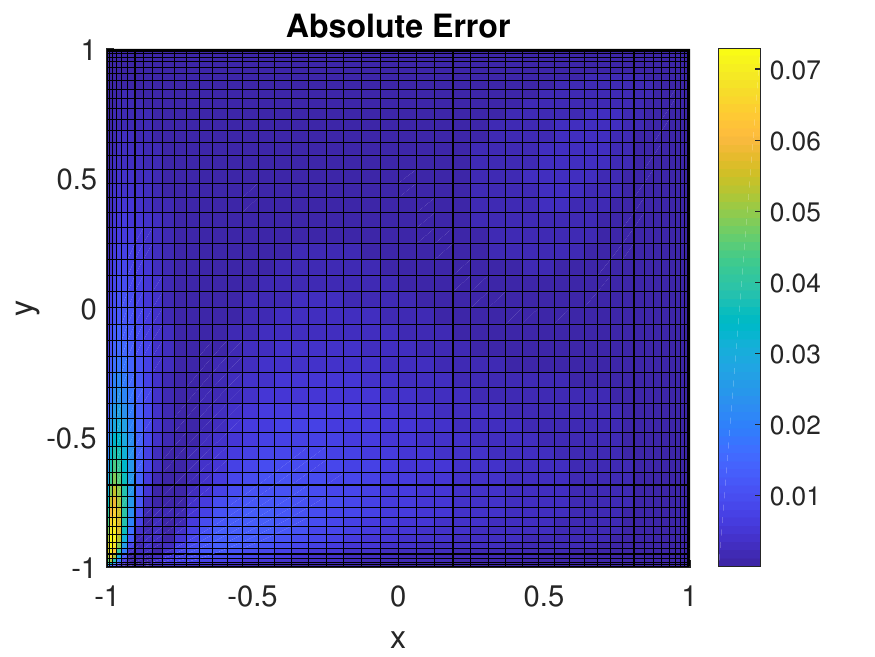} \\
    \caption{{Example \ref{ex:jiahua}. \mvpaaa approximations for two PDEs: left-pane for the PDE in~\eqref{eq:pde1} for $p=1.7545$ and $z=2$ and  right-pane for the PDE in~\eqref{eq:pde2} for $p=0.95$ and $z=0.99$)}
    }
    \label{fig:jiahuaWorst}
\end{figure}

\section{Conclusions}
\label{sec:conclusions}
We have presented a data-driven modeling framework for approximating parametric (dynamical) systems by extending the \aaa algorithm to multivariate problems. The method does not require access to an internal state-space description and works with function evaluations. We have discussed the scalar-valued problem as well as the matrix-valued ones. Various numerical examples have been used to illustrate the effectiveness of the proposed approach.  

\section*{Acknowledgements}
We thank Thanos Antoulas and Cosmin Ionita for providing their code for computing the parametric Loewner approximant. 
We also thank Vijaya Sriram Malladi
for providing the parametric beam model studied in \Cref{sec:sriram}.

\bibliographystyle{siamplain}
\bibliography{references}

\appendix
\section{State-space realization} \label{appendix:realization}
{
First, we recall the formulae derived in \cite{Io14data}, which allow for computing state-space realizations based on a given two-variable barycentric form. In the following, assume that the barycentric form \eqref{eq:Hrbary} computed by the \paaa algorithm is given. Define the parameter dependent terms
\begin{equation*}
\hat{\alpha}_i(p) = \sum_{j=1}^{q+1}\frac{\alpha_{ij}}{p-\pi_j} \quad \text{and} \quad \hat{\beta}_i(p) = \sum_{j=1}^{q+1}\frac{\beta_{ij}}{p-\pi_j},
\end{equation*}
as well as the system matrices 
    \begin{equation*}
    \label{eq:ssrealization}
    s\hbE - \hbA(p) = \left[ \begin{array}{cccc} 
    s-\sigma_1 & \sigma_2-s & & \\
    \vdots & & \ddots & \\
    s-\sigma_1 & & & \sigma_{k}-s \\
    \hat{\alpha}_1(p) & \hat{\alpha}_2(p) & \hdots & \hat{\alpha}_{k}(p)
    \end{array} \right],
    \hbb = \left[ \begin{array}{c} 0 \\ \vdots \\ 0 \\ 1\end{array} \right],
    \hbc(p) = \left[ \begin{array}{c} \hat{\beta}_1(p) \\ \hat{\beta}_2(p) \\ \vdots \\  \hat{\beta}_{k}(p) \end{array} \right].
    \end{equation*}
An equivalent representation to the barycentric form \eqref{eq:Hrbary} is then given by
    \begin{equation*}
    \widetilde{H} (s,p) = \hbc(p)^\top (s\hbE - \hbA (p))^{-1} \hbb.
    \end{equation*}
For a detailed discussion regrading the connection between the matrix pencil $s\hbE - \hbA(p)$ and the barycentric form we refer the reader to \cite{Io14data}. Note, that the presented matrices, their dimensions and the type of parameter dependence are not unique. For example, an equivalent realization without parameter dependence in $\hbc(p)$ but larger matrices was derived in \cite{An12two}.
}
{
\paragraph{Real system matrices}
Whenever complex-valued frequencies are used for generating transfer function samples, the matrix $\hbA(p)$ as well as $\hbc(p)$ are also complex-valued. In \cite{Io14data} the authors demonstrate that under the condition that interpolated complex frequencies exclusively appear as complex conjugate pairs, real-valued system matrices can be computed. Note that $H (\overline{s},p) = \overline{H (s,p)}$ since the underlying system is assumed to be real. This reveals that we can obtain samples from conjugates of complex frequencies without having to compute or measure additional values.
Consider the partitioning \eqref{eq:Lpart} and relabel the frequencies according to the previously mentioned condition:
\begin{equation} \label{eq:realcomplex}
\{s_1, \dots, s_N\} = \{ \s_1, \dots, \s_r, \s_{r+1}, \overline{\s}_{r+1}, \ldots, \s_{r+c}, \overline{\s}_{r+c}\} \cup 
		    \{ \sh_1, \dots, \sh_{N-k} \},
\end{equation}
where $\sigma_1,\ldots,\sigma_r$ are real-valued, $\sigma_{r+1},\ldots,\sigma_{r+c}$ are complex-valued and $k=r + 2c$. First, as done in \cite{Io14data}, consider the case $r \geq 1$ in \cref{eq:realcomplex} and define the matrices
\begin{equation*}
    \bJ = \frac{1}{\sqrt{2}}\left[ \begin{array}{cc}
         1 & 1 \\
         -i & i
    \end{array} \right], \; \bU = \left[ \begin{array}{ccc}
         \bI_{r-1} & & \\
         & \bI_c \otimes \bJ & \\\
         & & 1
    \end{array} \right], \; \bV = \left[ \begin{array}{cc}
         \bI_{r} & \\
         & \bI_c \otimes \bJ
    \end{array} \right].
\end{equation*}
A realization consisting only of real matrices is then given by $\hbb_r = \bU \hbb = \hbb$, $\hbc_r = \hbc \bV^*$ and $s\hbE_r - \hbA_r(p) = s \bU \hbE \bV^* - \bU \hbA(p) \bV^*$, such that
\begin{equation*}
    \widetilde{H} (s,p) = \hbc_r(p)^\top (s\hbE_r - \hbA_r (p))^{-1} \hbb_r.
\end{equation*}
Following \cite{Io14data}, one can write the real system matrices explicitly as
\begin{equation*}
    \hbc_r^\top = \left[\arraycolsep=1pt\begin{array}{c} \hat{\beta}_1(p),\ldots,\hat{\beta}_r(p), \Re(\hat{\beta}_{r+1}(p)),-\Im(\hat{\beta}_{r+1}(p)),\ldots,\Re(\hat{\beta}_{r+c}(p)),-\Im(\hat{\beta}_{r+c}(p))\end{array} \right]
\end{equation*}
and
\begin{equation*}
    s\hbE_r - \hbA_r(p) = \left[ \begin{array}{ccccccc} 
    s-\s_1 & \s_2-s & & & & &\\
    \vdots & & \ddots & & & & \\
    s-\s_1 & & & \s_r -s & & & \\
    \bff & & & & \bg_{r+1} \\
    \vdots & & & & & \ddots & \\
    \bff & & & & & & \bg_{r+c} \\ 
    \hat{\alpha}_{1}(p) & \hat{\alpha}_{2}(p) & \cdots & \hat{\alpha}_{r}(p) & \gamma_{r+1}(p) & \cdots & \gamma_{r+c}(p)
    \end{array} \right],
\end{equation*}
where $\bff = \left[ \begin{array}{c} s - \s_1 \\ 0 \end{array} \right]$, $\bg_i = \left[ \begin{array}{cc} \Re \s_i - s & - \Im \s_i \\ \Im \s_i & \Re \s_i - s \end{array} \right] $ and $\gamma_i^\top = \left[ \begin{array}{c} \Re \hat{\alpha}_i(p) \\ -\Im \hat{\alpha}_i(p) \end{array} \right]$.}

{
In our dynamical system examples, we have $r=0$ in \cref{eq:realcomplex}, i.e., we do not have a real frequency sample. We now provide some modifications to handle this case. For $r=0$, we define the vector $\ell = \frac{1}{\sqrt{2}}\left[ \begin{array}{ccc} 1 & \cdots & 1 \end{array} \right] \otimes \left[ \begin{array}{cc} -1 & 0 \end{array} \right]$ and 
\begin{equation*}
    \bU_0 = \left[ \begin{array}{ccc}
         \frac{i}{\sqrt{2}} & & \\[1ex]
         \ell^\top & \bI_{c-1} \otimes \bJ & \\
         & & 1
    \end{array} \right].
\end{equation*}
We then obtain
    \begin{equation*}
    s\hbE_c - \hbA_c(p) = s \bU_0 \hbE \bV^* - \bU_0 \hbA(p) \bV^* = \left[ \begin{array}{ccccc} 
    \tilde{\bff}_0 & & &\\
    \tilde{\bff} & \bg_2 & &\\
    \vdots & & \ddots &\\
    \tilde{\bff} & & & \bg_{c} \\ 
    \gamma_{1}(p) & \gamma_{2}(p) & \cdots & \gamma_{c}(p)
    \end{array} \right],
    \end{equation*}
where $\tilde{\bff}_0^\top = \left[ \begin{array}{c} \Im(\s_1) \\ \Re(\s_1) - s \end{array} \right]$ and $\tilde{\bff} = \left[ \begin{array}{cc} s-\Re(\s_1) & \Im(\s_1) \\ 0 & 0 \end{array} \right]$. Using $\hbc_r(p)$ and $\hbb_r$ from the previously discussed case we obtain the real state-space form
\begin{equation*}
    \widetilde{H} (s,p) = \hbc_r(p)^\top (s\hbE_c - \hbA_c (p))^{-1} \hbb_r.
\end{equation*}
}

\section{Minimal order interpolant}
\label{appendix:mininterpolant}
{A key result in \cite{Io14data} reveals that the minimal order of a two-variable rational interpolant is given by $(k_*,q_*)$, where
\begin{equation}
\label{eq:minorder}
    k_* = \max_{j=1,\ldots,M} \rk \Loew_{p_j} \quad \text{and} \quad q_* = \max_{i=1,\ldots,N} \rk \Loew_{s_i},
\end{equation}
and the 1D Loewner matrices $\Loew_{p_j}$ and $\Loew_{s_i}$ are defined in \eqref{Lpij} and \eqref{Lsigmailj}, respectively.
Moreover, any partitioning as in \eqref{eq:Lpart} with $k > k_*$ and $q > q_*$ yields a rational function that interpolates all function samples. In other words, one could compute a priori upper bounds for the order of the approximant by computing $M+N$ SVDs of 1D Loewner matrices and avoid constructing non-minimal interpolants when using the \paaa algorithm. Note that for large data sets this is a potentially expensive task. Instead, we propose an approach, which computes a minimal interpolant via a post-processing procedure. First, we  answer the question of how we can tell whether the output from \Cref{alg:p-AAA} is  a non-minimal interpolant or not without computing $k_*$ and $q_*$ as in \eqref{eq:minorder}. 
\begin{lemma} \label{lem:minorder}
Consider the data \eqref{eq:Lpart} and 
let the corresponding barycentric rational approximant $\Hr (s,p) $ have the form in \eqref{eq:Hrbary}. {Furthermore, for at least one $\tilde{p} \in \{p_j\}$ and $\tilde{s} \in \{ s_i \}$ satisfying $k_* = \rk \Loew_{\tilde{p}} $ and $q_* = \rk \Loew_{\tilde{s}}$, we assume that all $k_* \times k_*$ submatrices of $ \Loew_{\tilde{p}} $ and $q_* \times q_*$ submatrices of $\Loew_{\tilde{s}}$ have full rank.} Then
$$
\dim \ker \Loew_2 \geq 1 \quad \mbox{if~and~only~if} \quad k > k_* \text{ and } q > q_*.
$$
In addition, if $\dim \ker \Loew_2 \geq 1$, then $\dim \ker \Loew_2 = (k-k_*)(q-q_*)$.
\end{lemma}
\begin{proof}
We will only show the first implication $(\Rightarrow)$ and refer the reader to \cite{Io14data} for the proof of the other direction $(\Leftarrow)$. 
Let $\dim \ker \Loew_2 \geq 1$ and assume that $k \leq k_*$ or $q \leq q_*$. First, $\dim \ker \Loew_2 \geq 1$ implies that $\Hr (s,p) $ interpolates all data in \eqref{eq:Lpart} (this follows from the error formula derived in Corollary 4.3 in \cite{Io14data}). Let $\tilde{p}$ be a parameter where the first expression in \eqref{eq:minorder} attains its maximum. In other words $k_*= \rk \Loew_{\tilde{p}}$. Further, assume that all $k_* \times k_*$ submatrices of $\Loew_{\tilde{p}}$ have full rank. These conditions imply that a rational interpolant of the values $H(s_i,\tilde{p})$ for $i=1,\ldots,N$ has to be at least of order $k_*$ \cite{AntBG20}. However, $\Hr (s,\tilde{p}) $ interpolates all these points and is of order $k_*-1$ or less. A similar contradiction can be shown in the case that $q \leq q_*$ yielding that $k > k_*$ and $q > q_*$. Based on this result, we can apply Theorem 4.2. from \cite{Io14data} which implies that $\rk \Loew_2 = kq - (k-k_*)(q-q_*)$. Since $\Loew_2$ has $kq$ columns we obtain $\dim \ker \Loew_2 = (k-k_*)(q-q_*)$.
\end{proof}
\paragraph{Algorithmic implications}
\Cref{lem:minorder} reveals a connection between the nullity of $\Loew_2$ and the minimal order of an interpolant, {based on uniform rank conditions of 1D Loewner matrices that are typically satisfied in practice \cite{AntBG20}}. If $\dim \ker \Loew_2 = d = 1$ we have $k=k_*+1$ and $q=q_*+1$ and the interpolant is of minimal order. If $d>1$ it must be that $k>k_*+1$ or $q>q_*+1$ and the interpolant is of non-minimal order. Assuming that we compute the SVD of $\Loew_2$ using direct methods, $d$ is  available in each step of the algorithm without the need for additional computations. Our proposed post-processing procedure, which can be used after Line \ref{lst:line:lsq} of \Cref{alg:p-AAA} if $d>1$, is as follows: 
\begin{enumerate}
    \item If $q \leq k$ compute $q_*$ based on \eqref{eq:minorder} and if $k < q$ compute $k_*$ based on \eqref{eq:minorder}.
    \item \Cref{lem:minorder} implies that $k_*=k-d/(q-q_*)$ and $q_*=q-d/(k-k_*)$. From the first step we either obtain $k_*$ or $q_*$. In the former case we compute $q_*=q-d/(k-k_*)$ whereas in the latter case we compute $k_*=k-d/(q-q_*)$ in order to obtain the order of the minimal interpolant $(k_*,q_*)$.
    \item Update the partitioning \eqref{eq:Lpart} such that $k=k_*+1$ and $q=q_*+1$ and use it to compute $\Hr$ based on Lines \ref{lst:line:lsq} and \ref{lst:line:Hr} of \Cref{alg:p-AAA}.
\end{enumerate}
}

\end{document}